\documentclass[11pt]{article}
\usepackage{amssymb}
\usepackage{graphicx}
\usepackage{xcolor} 
\usepackage{tensor}
\usepackage{fullpage} 
\usepackage{amsmath}
\usepackage{amsthm}
\usepackage{verbatim}

\usepackage{enumitem}
\setlist[enumerate]{leftmargin=1.5em}
\setlist[itemize]{leftmargin=1.5em}

\usepackage{hyperref}
\usepackage{seqsplit,mathtools}

\mathtoolsset{showonlyrefs}

\definecolor{green}{rgb}{0,0.8,0} 

\newtheorem{thm}{Theorem}[section]

\newtheorem{lem}[thm]{Lemma}
\newtheorem{prop}[thm]{Proposition}

\newtheorem{defn}[thm]{Definition}

\theoremstyle{definition}

\theoremstyle{remark}
\newtheorem{rmk}[thm]{Remark}

\numberwithin{equation}{section}

\newcommand{\nnrm}[1]{{\vert\kern-0.25ex\vert\kern-0.25ex\vert #1 
		\vert\kern-0.25ex\vert\kern-0.25ex\vert}}

\begin{document}
\bibliographystyle{plain}
\title{Stability of radially symmetric, monotone vorticities of \\ 2D Euler equations}
\author{Kyudong Choi\thanks{Department of Mathematical Sciences, Ulsan National Institute of Science and Technology. Email: kchoi@unist.ac.kr}\and 
  Deokwoo Lim\thanks{Department of Mathematical Sciences, Ulsan National Institute of Science and Technology. Email: dwlim@unist.ac.kr}   }

\date\today
 
  \maketitle

\renewcommand{\thefootnote}{\fnsymbol{footnote}}
\renewcommand{\thefootnote}{\arabic{footnote}}

\begin{abstract}
We consider  the incompressible Euler equations in $ \mathbb{R}^{2} $ when the initial vorticity  is bounded, radially symmetric and non-increasing in the radial direction. Such a radial distribution  is stationary, and we show that the monotonicity produces  stability in some weighted  norm related to the angular impulse.
For instance, it covers the cases of circular vortex patches and Gaussian distributions.
Our   stability does not depend on  
 $L^\infty$-bound  or  support size of   perturbations. The proof is based on the fact that such a radial monotone distribution minimizes the  impulse of functions having the same level set measure.
\end{abstract}\vspace{1cm} 
{
 \section{Introduction}}
\noindent  
 
Stability   for  special solutions of the incompressible Euler equations has been an important topic in hydrodynamics since Kelvin's work \cite{Kelvin1880col}.
In this paper, we are interested in stability    of  circular flows generated by \textit{radial, monotone} vorticities. 
 
We consider the incompressible Euler equations in $ \mathbb{R}^{2} $ in vorticity form:
\begin{align}
	\begin{split}
		\partial_t\omega+u\cdot\nabla\omega&=0\quad\text{ for }\quad(t,x)\in(0,\infty)\times\mathbb{R}^{2},\\
		\omega|_{t=0}&=\omega_{0}\quad\text{ for }\quad x\in\mathbb{R}^{2},\label{eq1}
	\end{split}	
\end{align}
where the velocity $ u $ is determined from the vorticity $ \omega $ by the Biot-Savart law given as
\begin{equation}
	\nonumber u(x)=\int_{\mathbb{R}^{2}}K(x-y)\omega(y)dy,\quad K(x)=\frac{1}{2\pi}\left( -\frac{x_{2}}{|x|^{2}},\frac{x_{1}}{|x|^{2}}\right) .\label{eq3}
\end{equation}  
Yudovich's theory \cite{Yu63} says that if the initial vorticity $ \omega_{0} $ lies on $ (L^{1}\cap L^{\infty})(\mathbb{R}^{2}) $, then there exists a unique global-in-time weak solution $ \omega(t) $ of \eqref{eq1}. For modern treatments of the theory, we refer to  Majda-Bertozzi \cite[Ch. 8]{MaBVo}, Marchioro-Pulvirenti \cite[Ch. 2]{MaPMT},  Chemin \cite[Ch. 5]{ChPIF}. We denote the angular impulse of a function $ f $ by
\begin{equation}
	\nonumber J(f):=\int_{\mathbb{R}^{2}}|x|^{2}f(x)dx.
\end{equation}
For $ p\in [1,\infty)$, we define the norm $ \left\| \cdot\right\| _{{J}_{p}} $ by
\begin{gather}
	\nonumber\left\| g\right\|_{{J}_{p}}:=\left\| g\right\| _{L^{p}(\mathbb{R}^2)}+J(|g|)=\left( \int_{\mathbb{R}^{2}}|g(x)|^{p}dx\right) ^{\frac{1}{p}}+\int_{\mathbb{R}^{2}}|x|^{2}|g(x)|dx,
\end{gather}
for $ g\in L^{p}(\mathbb{R}^{2}) $ that satisfies $ J(|g|)<\infty $. 
We remind readers that
 for any nonnegative initial data $ \omega_{0}\in   L^{\infty}(\mathbb{R}^{2}) $ with $ J(\omega_{0})<\infty $, the corresponding solution $\omega(t)$ of \eqref{eq1}  is nonnegative and has conservations of any $ L^{p} $-norm $\|\omega(t) \|_{L^p}$ for $ p\in[1,\infty] $ and the angular impulse $J(\omega(t))$. In particular,  the Lebesgue measure of each level set  
\begin{equation}\label{intro_level}
  |\{x\in\mathbb{R}^2\,:\,\omega(t,x)>\alpha\}|, \quad \alpha>0
\end{equation}  is preserved in time. 

Throughout this paper, we denote $ B_{r}  $ as the disk centered at the origin with radius $ r>0 $;
\begin{equation}
	\nonumber B_{r}:=\left\lbrace x\in\mathbb{R}^{2} : |x|<r\right\rbrace.
\end{equation} For convention, 
we also denote 
\begin{equation}
	\nonumber B_{0}:=\mbox{the empty set},\quad D:= \mbox{the unit disk } B_{1}.
\end{equation}

\subsection{
Main results}

It is a well-known fact that the characteristic function $ 1_{D} $ of the unit disk is a stationary solution of \eqref{eq1}. As a pioneer work, Wan-Pulvirenti \cite{WaP85}  obtained $ L^{1} $-stability of $ 1_{D} $ in  patch-type perturbations  for circular bounded domains $B_R,\, R>1$. For the whole space $\mathbb{R}^2$,   Sideris-Vega \cite{SiV09} proved  $ L^{1} $-stability with the explicit estimate
\begin{equation}
	\nonumber\sup_{t\geq0}| {\Omega_{t}}\bigtriangleup{D}|^2\leq  {4\pi\sup_{\Omega_{0}\bigtriangleup D}\left| |x|^{2}-1\right|\cdot |{\Omega_{0}}\bigtriangleup{D}|},
\end{equation}
where  $1_{\Omega_t}$ is the solution for   patch-type initial data $1_{\Omega_0}$, and the symbol $\bigtriangleup$ means the symmetric difference.
 We also refer to the result of Dritschel \cite{Dr88}.
Classical numerical results around $1_D$ can be found in  Deem-Zabusky \cite{DeZ78} and Buttke \cite{Bu90}. We also see \cite{CJ_perimeter}, \cite{Choi_distance}, \cite{CJ_wind}, Elgindi-Jeong \cite{EJSVP2}
for some applications about growth in perimeter of patch boundary and  winding number of particles. \\

First, we obtain $ {J}_{2} $-stability of $1_D$ allowing nonpatch-type perturbations which are not necessarily compactly supported.
\begin{thm}\label{thm1}
	For $ \varepsilon>0 $, there exists $ \delta>0 $ such that if a nonnegative $ \omega_{0}\in L^{\infty}(\mathbb{R}^{2}) $ with $ J(\omega_{0})<\infty $ satisfies
	\begin{equation}
		\nonumber \left\| \omega_{0}-1_{D}\right\| _{{J}_{2}}\leq\delta,
	\end{equation}
	then the solution $ \omega(t) $ of \eqref{eq1} satisfies
	\begin{equation}
		\nonumber \sup_{t\geq0}\left\| \omega(t)-1_{D}\right\| _{{J}_{2}}\leq\varepsilon.
	\end{equation}
\end{thm}

Such  stability lies on  the fact that $ 1_{D} $ \textit{uniquely} minimizes the angular impulse $J$ of nonnegative functions $\xi$ bounded by $ 1 $ and having the same $ L^{1} $-norm as $ 1_{D} $ (see Proposition \ref{prop1-1});
\begin{equation*}
0\leq \xi\leq 1, \quad \|\xi\|_{L^1(\mathbb{R}^2)}=\|1_D\|_{L^1(\mathbb{R}^2)}.  
\end{equation*}
 During the proof,  $ \delta $ is implicitly obtained by a contradiction argument.

Now we consider a generalization of circular patches. We first note that every radially symmetric vorticity is also a stationary solution of \eqref{eq1} because it makes the stream function $ \phi=\Delta^{-1}\omega $ be radially symmetric as well, which gives us
\begin{equation}
	\nonumber u\cdot\nabla\omega=\nabla^{\perp}\phi\cdot\nabla\omega=0
\end{equation}
(cf. 
see G\'{o}mez Serrano-Park-Shi-Yao
\cite{JPSY} for a related   converse  statement).
Let's take any  function $ \zeta\in L^{\infty}(\mathbb{R}^{2}) $ which is \textit{nonnegative, radially symmetric, and non-increasing}; there exists a function $ f\in L^{\infty}([0,\infty)) $ such that
\begin{equation}
	\nonumber \zeta(x)=f(|x|)\quad\text{ for }\quad x\in\mathbb{R}^{2},\quad f\geq0,\quad f(r_{1})\geq f(r_{2})\quad\text{ for }\quad r_{1}\leq r_{2}.
\end{equation}
Theorem \ref{thm2} below says that $ \zeta $ is $ {J}_{p} $-stable for $ p\in[1,\infty) $ with an explicit estimate \eqref{estimate} whenever $ \zeta $ is compactly supported in $ \mathbb{R}^{2} $. For instance, Theorem \ref{thm1} is just a particular case ($\zeta=1_D$ and $p=2$) of Theorem \ref{thm2}. 

\begin{thm}\label{thm2}
	($ \text{supp }(\zeta) $ : compact) For any constants $ R,M>0$, there exists a constant $ C=C(R,M) >0 $ such that if $ \zeta\in L^{\infty}(\mathbb{R}^{2}) $ is nonnegative, radially symmetric, non-increasing, and compactly supported with
	\begin{equation}
		\text{supp }(\zeta)\subset B_{R},\quad \left\| \zeta\right\| _{L^{\infty}}\leq M,\label{RM}
	\end{equation}
	then for any nonnegative $ \omega_{0}\in L^{\infty}(\mathbb{R}^{2}) $ with $ J(\omega_{0})<\infty $, the solution $ \omega(t) $ of \eqref{eq1} satisfies
	\begin{equation}
		\sup_{t\geq0}\left\| \omega(t)-\zeta\right\| _{{J}_{p}}\leq C\left[ \left\| \omega_{0}-\zeta\right\| _{{J}_{p}}^{\frac{1}{2p}}+\left\| \omega_{0}-\zeta\right\| _{{J}_{p}}\right] \quad\text{ for any }\quad p\in[1,\infty).\label{estimate}
	\end{equation}
\end{thm}

As  in the case $1_D$, the key idea is that such a profile $\zeta$ 
minimizes the angular impulse $J$ of nonnegative functions $\xi$ having
 the same level set measure as $\zeta$;
$$|\{\xi>\alpha\}| =|\{\zeta>\alpha\}| ,\quad \alpha>0.$$
   We note that Marchioro-Pulvirenti \cite{MaP85} already showed $L^1-$stability of such monotone profiles  for any circular bounded domains $ B_{R} $, $ R>0$ (also see Burton  \cite[Theorem 3]{Burton2005}). The paper \cite{MaP85}   also showed 
   $ L^{1} $-stability  of  any monotone function $ \zeta(x,y)=\zeta(y) $   for bounded channels $ \mathbb{T}\times[0,R] $, $ R>0$.

Lastly, we consider the case when a monotone profile $\zeta$ is not necessarily compactly supported. In this case, we request, instead,  the profile $\zeta$ to have finite momentum of some higher order \eqref{assu_high}.
For instance, we obtain stability of a vorticity with the Gaussian distribution $e^{-|x|^2}$. 
Bassom-Gilbert \cite[Section 4]{Bassom} examined   the Gaussian profile numerically, which was motivated by   Lamb \cite[334a]{Lamb6th}. We also see 
Schecter-Dubin-Cass-Driscoll-Lansky-O'Neil
 \cite{Schecter} and references therein for linear analysis and   experiments on radial profiles. 

\begin{thm}\label{thm3}
	($ \text{supp }(\zeta) $ : not necessarily compact) Let $ \zeta\in L^{\infty}(\mathbb{R}^{2}) $ be nonnegative, radially symmetric, and non-increasing with
\begin{equation}\label{assu_high}
 \int_{\mathbb{R}^{2}}|x|^{6}\zeta dx<\infty,
\end{equation}		
	and let $ p\in[1,\infty) $. Then for $ \varepsilon>0 $, there exists $ \delta=\delta(\varepsilon,\zeta,p)>0 $ such that if a nonnegative $ \omega_{0}\in L^{\infty}(\mathbb{R}^{2}) $ with $ J(\omega_{0})<\infty $ satisfies
	\begin{equation}
		\nonumber \left\| \omega_{0}-\zeta\right\| _{{J}_{p}}\leq\delta,
	\end{equation}
	then the solution $ \omega(t) $ of \eqref{eq1} satisfies
	\begin{equation}
		\nonumber\sup_{t\geq0}\left\| \omega(t)-\zeta\right\| _{{J}_{p}}\leq\varepsilon.
	\end{equation}
\end{thm}
\begin{rmk}
All our Theorems \ref{thm1}, \ref{thm2}, \ref{thm3} ask initial vorticity  $\omega_0$ to be \textit{nonnegative}. Removing the sign assumption looks highly non-trivial as long as we consider fluids in the whole space $\mathbb{R}^2$ (cf.  no sign condition results \cite{MaP85}, \cite{Burton2005} for  circular bounded domains). This is because 
our stability is essentially due to  the conservation of the  fluid  impulse $\int_{\mathbb{R}^2}|x|^2\omega(t,x)\,dx$.
Indeed,   when allowing negative part of $\omega$,  it might be possible   that some negative region might fly away (toward infinity) in time
(cf. see the example in 
Iftimie-Sideris-Gamblin
\cite[Sec. 3]{ISG}). In other words, we do not know a global-in-time bound on
 $\int_{\{\omega<0\}}|x|^2\omega(t,x)\,dx$.
Technically speaking, 
for a function on
the \textit{whole} space $\mathbb{R}^2$,  
the  {nonnegativity} is needed 
to define its  symmetric-decreasing rearrangement (see Definition \ref{defn1}).

\end{rmk}
For strictly monotone radial profiles in $\mathbb{R}^2$, Bedrossian-Coti Zelati-Vicol \cite{BeCZV19}  obtained the linear inviscid damping. 
We also mention asymptotic results for a point vortex by Coti Zelati-Zillinger \cite{CotiZill} and by
Ionescu-Jia \cite{IonJia}.
For the  infinite channel  $\mathbb{T}\times\mathbb{R}$, Beichman-Denisov \cite{BeD17} obtained stability for long rectangular patches. It is interesting that for      Couette flows in the same channel,    Bedrossian-Masmoudi \cite{BesMas} proved even the inviscid damping with   nonlinear asymptotic stability. For a bounded channel, we refer to Ionescu-Jia \cite{IonJia_cmp}.

Orbital  stability   of  other various vorticities 
can be found \textit{e.g.} in 
Tang \cite{Tang} on Kirchhoff’s ellipse, 
 Burton-Nussenzveig Lopes-Lopes Filho \cite{BNL13} on various dipoles,
   Cao-Wan-Wang \cite{CWW19} on patches for bounded domains,
\cite{AbC19} on   Lamb (Lamb-Chaplygin) dipole, and \cite{Ch20} on   Hill's spherical  vortex.

\subsection{
Key ideas}

 We   follow the strategy via the variational method based on \textit{vorticity} due to
 the classical papers by Kelvin \cite{Kelvin1880} and Arnold \cite{Arnold66} (also see the book of Arnold-Khesin \cite{AK98}). 
To go into detail and explain, the stability in this paper is closely related to the following properties of the symmetric-decreasing rearrangement of functions. The first is that the angular impulse of a nonnegative function $ f $ is greater than or equal to that of the rearrangement $ f^{\ast} $ (see Definition \ref{defn1}) of $ f $;
\begin{equation}
	J(f^{\ast})\leq J(f).\label{Jbasic}
\end{equation}
It is simply because the weight $ |x|^{2} $ is radially symmetric and monotonically increasing.  
On the other hand, any two functions having the same measure for each level set have the same rearrangement. By this, the rearrangement of the solution $ \omega(t) $ of the Euler flow stays the same as that of $ \omega_{0} $ throughout all time;
\begin{equation}
	\nonumber\Big(\omega(t)\Big)^{\ast}=(\omega_{0})^{\ast}\quad\text{ for all }\quad t\geq0.
\end{equation} 

The second is the nonexpansivity (Lemma 
\ref{lemnonexp})
of rearrangements for nonnegative functions;
	\begin{equation*}
		\left\| g^{\ast}-h^{\ast}\right\| _{L^{1}}\leq\left\| g-h\right\| _{L^{1}}.
	\end{equation*} It says that $L^1$-difference between two functions is non-increasing by replacing them with  their rearrangements. It is one of well-known properties of rearrangements. For instance, a proof can be found in Lieb-Loss
\cite[Sec. 3.5]{LiLAn}. 

The third is    the following estimate (Lemma \ref{lem2}) which is a sharper version of  \eqref{Jbasic}:
			\begin{equation*}
		\left\| f-f^{\ast}\right\| _{L^{1}}^{2}\leq4\pi\left\| f\right\| _{L^{\infty}}\bigg[ J(f)-J(f^{\ast}) \bigg].
	\end{equation*}
It  means that $ L^{1} $-difference between a function and its rearrangement can be controlled by the difference between their angular impulses.
Such a rearrangement estimate   appeared  in the paper \cite[Lemma 1]{MaP85} for bounded channel domains $\mathbb{T}\times[0,R],\, R>0$. For other fine properties of   rearrangements which are applicable to the Euler equations,
we   refer to a series of  works  from  Burton 
\cite{Burton1987}, \cite{Burton1989}, Burton-McLeod \cite{BurtonMcLeod} for bounded domains and  
Douglas \cite{Douglas} for unbounded domains.


We combine the above properties of rearrangements with the idea of cutting off a function to avoid using its supremum. 
It makes us to remove
the dependence on $ \left\| \omega_{0}\right\| _{L^{\infty}} $ in $ L^{1} $-estimate (Lemma \ref{lem3}). In addition, we use the fact that the rearrangement operation and the cut-off operation commute with each other; the rearrangement of a cut-off of a function is the same as the cut-off of the rearrangement.

For the organization of the rest, in Section 2, we prove Theorem \ref{thm1} via a contradiction argument. It is done by studying 
  $$\mbox{1. Existence: Lemma \ref{lem1-1}, \quad 2. Uniqueness: Proposition \ref{prop1-1}, \quad 3. Compactness: Proposition \ref{prop2}}$$
  of the variational problem \eqref{Var2}. 
In Section 3, we first show 
 $ L^{1} $-estimate \eqref{L1} in Lemma \ref{lem3}, which gives
 $ {J}_{p} $-estimate \eqref{mainestimate} in Lemma \ref{lemmain}. Then we have 
Theorems \ref{thm2} and \ref{thm3}.

Lastly, we recall again that Theorem \ref{thm1} is just a particular case of Theorem \ref{thm2}. Nevertheless, we decided to keep the proof of Theorem \ref{thm1}  in Section 2 because the variational method and the contradiction argument are more robust when  applying to more complicated settings
such as Lamb (Lamb-Chaplygin) dipole \cite{AbC19} or Hill's spherical vortex \cite{Ch20}.
However,  the constructive computations done in Section 3 seem not applicable to those complicated cases since they heavily rely on the explicit rearrangement estimate \eqref{keyestimate} in Lemma \ref{lem2}.  If one wants to cut to the chase toward Theorems \ref{thm2} and   \ref{thm3}, we recommend    to  skip Section 2  and to read Section 3.

We finish the introduction section by mentioning the \textit{easiest}   stability \eqref{simplest} of a circular patch for patch-type perturbations 
since 
it does not seem to be well-known in the community.  The current form \eqref{simplest} is due to 
\cite{SiV09} while \cite{Dr88}  also stated it in a convenient integral form.
For any patch-type solution $1_{\Omega_{t}}$ of \eqref{eq1},  the nonnegative quantity
$$ \int_{\Omega_{t}\bigtriangleup D}\left| |x|^{2}-1\right| dx$$ is conserved in time by the  direct computation:
\begin{equation}\begin{split}\label{simplest}
	 \int_{\Omega_{t}\bigtriangleup D}\left| |x|^{2}-1\right| dx&=\int_{\Omega_{t}\setminus D} (|x|^{2}-1) dx+\int_{D\setminus\Omega_{t}} (1-|x|^{2}) dx=\int_{\Omega_{t}} (|x|^{2}-1) dx+\int_{D} (1-|x|^{2}) dx\\
	 &=\int_{\Omega_{0}} (|x|^{2}-1) dx+\int_{D} (1-|x|^{2}) dx=\int_{\Omega_{0}\bigtriangleup D}\left| |x|^{2}-1\right| dx,\quad t\geq0.
\end{split}\end{equation}

\section{Variational method}


In this section, we use the notion of the symmetric-decreasing rearrangement of a function. To define this rearrangement rigorously, we first need the notion of the symmetric rearrangement of a measurable set with finite measure. Below we follow the definition in 
  the textbook \cite[Sec. 3.3]{LiLAn}.
\begin{defn}\label{defn1}
	(i) Let $ A\subset\mathbb{R}^{2} $ be a measurable set with $ |A|<\infty $. Then the symmetric rearrangement $ A^{\ast} $ of $ A $ is defined as $ B_{r} $ for some $ r\geq0 $ having the same measure as $ A $;
	\begin{equation}
		\nonumber A^{\ast}:= B_{r},\quad \pi r^{2}=|A|.
	\end{equation}
	
	(ii) Let $ f\in L^{p}(\mathbb{R}^{2}) $ for some $ p\in[1,\infty) $. Then the symmetric-decreasing rearrangement $ f^{\ast} $ of $ f $ is defined as the following;
	\begin{equation}
		\nonumber f^{\ast}(x):=\int_{0}^{\infty}1_{\left\lbrace y\in\mathbb{R}^{2} : |f(y)|>\xi\right\rbrace ^{\ast}}(x)d\xi,\quad x\in\mathbb{R}^{2}.
	\end{equation}
\end{defn}

\begin{rmk}
	It is easy to see that $ f^{\ast} $ is nonnegative, radially symmetric, non-increasing, and that each level set of $ f^{\ast} $ is the symmetric rearrangement of the level set of $ |f| $;
	\begin{equation}
		\left\lbrace x\in\mathbb{R}^{2} : f^{\ast}(x)>\alpha\right\rbrace =\left\lbrace x\in\mathbb{R}^{2} : |f(x)|>\alpha\right\rbrace ^{\ast}\quad\text{ for all }\quad \alpha>0.\label{levelset}
	\end{equation}
	So, we have
	\begin{equation}
		\left| \left\lbrace x\in\mathbb{R}^{2} : f^{\ast}(x)>\alpha\right\rbrace \right| =\left| \left\lbrace x\in\mathbb{R}^{2} : |f(x)|>\alpha\right\rbrace \right| \quad\text{ for all }\quad \alpha>0.\label{levelsetmeas}
	\end{equation}
	Also, each $ L^{q} $-norm of $ f^{\ast} $ for $ q\in[1,p] $ is preserved;
	\begin{equation}
		\left\| f^{\ast}\right\| _{L^{q}}=\left\| f\right\| _{L^{q}}.\label{Lp_preserve}
	\end{equation}
From now on, we simply call $ A^{\ast} $ and $ f^{\ast} $ the rearrangement of $ A $ and $ f $, respectively.
	
\end{rmk}

\begin{rmk}
	If a nonnegative $ f\in L^{\infty}(\mathbb{R}^{2}) $ satisfies $ J(f)<\infty $, as in the setting of our theorems, then we have $ f\in L^{p}(\mathbb{R}^{2}) $ for every $ p\in[1,\infty] $ because of
	\begin{equation*}
		\left\| f\right\| _{L^{1}}= \int_{D}fdx+\int_{\mathbb{R}^{2}\setminus D}fdx\leq\pi\left\| f\right\| _{L^{\infty}}+J(f),\label{L1bd}
	\end{equation*}
	and
	\begin{equation}
		\nonumber\left\| f\right\| _{L^{p}}^{p}\leq\left\| f\right\| _{L^{\infty}}^{p-1}\left\| f\right\| _{L^{1}},\quad p\in[1,\infty).
	\end{equation}
	So the rearrangement $ f^{\ast} $ is well-defined.
\end{rmk}
	
In this section, the basic property \eqref{Jbasic} of $ f^{\ast} $ is mainly used.

\subsection{Existence and uniqueness}

We introduce two admissible classes and the corresponding variational problems.

\begin{defn}\label{def1-1}
	We define admissible classes of functions by
	\begin{gather}
		\nonumber P=\left\lbrace f\in L^{1}(\mathbb{R}^{2}) : f=1_{\Omega}\quad\text{ for some measurable }\quad\Omega\subset\mathbb{R}^{2},\quad J(f)<\infty,\quad \left\| f\right\| _{L^{1}}=\pi\right\rbrace,\\
		\nonumber P'=\left\lbrace f\in L^{1}(\mathbb{R}^{2}) : 0\leq f\leq1,\quad J(f)<\infty,\quad \left\| f\right\| _{L^{1}}=\pi\right\rbrace.
	\end{gather}
	We also set variational problems
	\begin{gather}
		I=\inf_{f\in P}J(f),\label{Var1}\\
		I'=\inf_{f\in P'}J(f),\label{Var2}
	\end{gather}
	and denote sets of minimizers of the above problems by
	\begin{gather}
		\nonumber S=\left\lbrace f\in P : J(f)\leq J(g)\quad\text{ for all }\quad g\in P\right\rbrace ,\\
		\nonumber S'=\left\lbrace f\in P' : J(f)\leq J(g)\quad\text{ for all }\quad g\in P'\right\rbrace.
	\end{gather}
\end{defn}

\begin{rmk}
	Since the weight $ |x|^{2} $ of $ J $ is radially symmetric and monotonically increasing, it is obvious that the variational problem \eqref{Var1} has the unique minimizer $ 1_{D} $, that is,
	\begin{equation}
		S=\left\lbrace 1_{D}\right\rbrace \quad\text{ and }\quad I=J(1_{D})=\frac{\pi}{2}.\label{S=1D}
	\end{equation}
\end{rmk}

The following lemma says that the characteristic function $ 1_{D} $ of the unit disk is a minimizer even in the larger class $ P' $.

\begin{lem}[Existence]\label{lem1-1}
	\begin{equation}
		\nonumber I=I'\quad\text{ and }\quad 1_{D}\in S'.
	\end{equation}
\end{lem}
This lemma implies that $ S $ is contained in $ S' $.

\begin{proof}
	By \eqref{S=1D} and by $ P\subset P' $, it's enough to show
	\begin{equation}
		J(1_{D})\leq J(f)\quad\text{ for all }\quad f\in P'.\label{J1D<=f}
	\end{equation}	
	Without loss of generality, we may assume $ f=f^{\ast} $, thanks to the basic property \eqref{Jbasic}.
	
	To show \eqref{J1D<=f}, for each $ n\in\mathbb{N} $, we define level sets $ A_{k}^{(n)} $ of $ f $ for $ k=1,\cdots,n $ by
	\begin{equation}
		A_{k}^{(n)}=\left\lbrace x\in\mathbb{R}^{2} : f(x)>\frac{k}{n}\right\rbrace,\label{FkGk}
	\end{equation}
	and set the simple function $ g^{(n)} $ by
	\begin{equation}
		g^{(n)}=\sum_{k=1}^{n-1}\frac{1}{n}1_{A_{k}^{(n)}}.\label{fngn}
	\end{equation}
	Then it forms a sequence of simple functions $ \left\lbrace g^{(n)}\right\rbrace _{n=1}^{\infty} $ that is dominated by $ f $ and converges to $ f $ pointwise. Thus, $ g^{(n)} $ satisfies
	\begin{equation}
		\left\| g^{(n)}\right\| _{L^{1}}\longrightarrow\left\| f\right\| _{L^{1}}=\pi\quad\text{ as }\quad n\longrightarrow\infty.\label{gnpi}
	\end{equation}
	For each $ n\in\mathbb{N} $, we set $ \overline{r}^{(n)}\geq0 $ by
	\begin{equation}
		|B_{\overline{r}^{(n)}}|=\left\| g^{(n)}\right\| _{L^{1}}.\label{Brngn}
	\end{equation}
	Then by \eqref{gnpi}, we have
	\begin{equation}
		\nonumber \overline{r}^{(n)}\longrightarrow 1\quad\text{ as }\quad n\longrightarrow\infty.
	\end{equation}
	We claim for each $ n $, we have
	\begin{align}
		J(1_{B_{\overline{r}^{(n)}}})\leq J(g^{(n)}).\label{1Brn}		
	\end{align}
	Once this claim is shown, then taking $ n\longrightarrow\infty $ on both sides of \eqref{1Brn} gives us \eqref{J1D<=f}.
	
	Let's prove the above claim. We fix $ n\in\mathbb{N} $ and for simplicity, we drop the parameter $ n $;
	\begin{equation}
		\nonumber g=g^{(n)},\quad A_{k}=A_{k}^{(n)},\quad \overline{r}=\overline{r}^{(n)}.
	\end{equation}	
	Because of the property \eqref{levelset} and $ f=f^{\ast} $ being radially symmetric and non-increasing, we have
	\begin{gather}
		\nonumber A_{k}\supset A_{k+1},\quad k=1,\cdots,n-1,\\
		\nonumber A_{k}= B_{s_{k}}\quad\text{ for some }\quad s_{k}\geq0\quad\text{ satisfying }\quad\pi s_{k}^{2}=|A_{k}|,\quad k=1,\cdots n.
	\end{gather}
	In particular, we have $ s_{n}=0 $ because $ A_{n} $ is the empty set. Furthermore, $ g $ can be rewritten as
	\begin{equation}
		g=\sum_{k=1}^{n-1}h_{k},\label{g=sumhk}
	\end{equation}
	where $ h_{k} $ for each $ k=1,\cdots,n-1 $ is given as
	\begin{equation}
		\nonumber h_{k}=\frac{k}{n}1_{A_{k}\setminus A_{k+1}}=\frac{k}{n}1_{\left\lbrace y\in\mathbb{R}^{2} : s_{k+1}\leq|y|<s_{k}\right\rbrace }.
	\end{equation}
	Then $ h_{k} $ has magnitude $ \frac{k}{n}\leq1 $ and satisfies
	\begin{equation}
		\nonumber\left\| h_{k}\right\| _{L^{1}}=\frac{k}{n}\int_{A_{k}\setminus A_{k+1}}1dx=\frac{\pi k}{n}(s_{k}^{2}-s_{k+1}^{2}), 
	\end{equation}
	\begin{equation}
		\nonumber J(h_{k})=\frac{k}{n}\int_{A_{k}\setminus A_{k+1}}|x|^{2}dx=\frac{2\pi k}{n}\int_{s_{k+1}}^{s_{k}}r^{3}dr=\frac{\pi k}{2n}(s_{k}^{4}-s_{k+1}^{4}).
	\end{equation}
	Additionally, we define a function $ h_{k}' $ by
	\begin{equation}
		\nonumber h_{k}'=1_{\left\lbrace y\in\mathbb{R}^{2} : s_{k+1}\leq|y|<c_{k}\right\rbrace },
	\end{equation}
	where $ c_{k}\geq0 $ is chosen to satisfy
	\begin{equation}
		\left\| h_{k}\right\| _{L^{1}}=\left\| h_{k}'\right\| _{L^{1}},\label{hkL1=hk'L1}
	\end{equation}
	that is,
	\begin{equation}
		\nonumber\frac{k}{n}\left| \left\lbrace y\in\mathbb{R}^{2} : s_{k+1}\leq|y|<s_{k}\right\rbrace \right| =\left| \left\lbrace y\in\mathbb{R}^{2} : s_{k+1}\leq|y|<c_{k}\right\rbrace \right|.
	\end{equation}
	More specifically, we have
	\begin{equation}
		\frac{\pi k}{n}(s_{k}^{2}-s_{k+1}^{2})=\pi(c_{k}^{2}-s_{k+1}^{2}),\label{ckrelation}
	\end{equation}
	which gives us
	\begin{equation}
		c_{k}=\sqrt{\left( 1-\frac{k}{n}\right)s_{k+1}^{2}+\frac{k}{n}s_{k}^{2} }.\label{ck}
	\end{equation}
	We observe that $ h_{k}' $ is the function which has greater magnitude compared to $ h_{k} $, the annulus with smaller outer radius, yet has the same inner radius and $ L^{1} $-norm. Then we have
	\begin{equation}
		J(h_{k})\geq J(h_{k}').\label{Jhk>=Jhk'}
	\end{equation}
	Indeed, this holds because using \eqref{ckrelation} and \eqref{ck}, we have
	\begin{align}
		\nonumber J(h_{k})-J(h_{k}')&=\frac{\pi k}{2n}(s_{k}^{4}-s_{k+1}^{4})-\frac{\pi}{2}(c_{k}^{4}-s_{k+1}^{4})\\
		\nonumber&=\frac{\pi k}{2n}(s_{k}^{2}+s_{k+1}^{2})(s_{k}^{2}-s_{k+1}^{2})-\frac{1}{2}\left[ \left(1-\frac{k}{n} \right) s_{k+1}^{2}+\frac{k}{n}s_{k}^{2}+s_{k+1}^{2}\right]\frac{\pi k}{n} (s_{k}^{2}-s_{k+1}^{2})\\
		\nonumber&=\frac{\pi k}{2n}\left( 1-\frac{k}{n}\right) (s_{k}^{2}-s_{k+1}^{2})^{2}\geq0.
	\end{align}
	Since $ h_{k}' $ is a characteristic function of an annulus for each $ k=1,\cdots,n-1 $, the rearrangement of the summation $ \left( \sum_{k=1}^{n-1}h_{k}'\right)  $ is the characteristic function of some disk; there exists $ \widehat{r}=\widehat{r}(n)\geq0 $ such that
	\begin{equation*}
		1_{B_{\widehat{r}}}=\left( \sum_{k=1}^{n-1}h_{k}'\right) ^{\ast}.\label{hksumstar}
	\end{equation*}
	Since the annuli $ \left\lbrace s_{k+1}\leq|x|<s_{k}\right\rbrace  $ are pairwise disjoint and the annuli $ \left\lbrace s_{k+1}\leq|x|<c_{k}\right\rbrace  $ satisfy the same as well, we get $ \overline{r}=\widehat{r} $, where $ \overline{r} $ is from the equation \eqref{Brngn}. Indeed, together with the preservation property \eqref{Lp_preserve} and the equation \eqref{hkL1=hk'L1}, we have
	\begin{align}
		\nonumber|B_{\overline{r}}|=\left\| g\right\| _{L^{1}}&=\left\| \sum_{k=1}^{n-1}h_{k}\right\| _{L^{1}}=\sum_{k=1}^{n-1}\left\| h_{k}\right\| _{L^{1}}=\sum_{k=1}^{n-1}\left\| h_{k}'\right\| _{L^{1}}=\left\| \sum_{k=1}^{n-1}h_{k}'\right\| _{L^{1}}=\left\| \left( \sum_{k=1}^{n-1}h_{k}'\right) ^{\ast}\right\| _{L^{1}}=|B_{\widehat{r}}|.
	\end{align}
	Finally by the basic property \eqref{Jbasic}, the equation \eqref{g=sumhk}, \eqref{Jhk>=Jhk'}, and the linearity of $ J $, we obtain \eqref{1Brn};
	\begin{align}
		\nonumber J(g)&=J\left( \sum_{k=1}^{n-1}h_{k}\right) =\sum_{k=1}^{n-1}J(h_{k})\geq\sum_{k=1}^{n-1}J(h_{k}')=J\left( \sum_{k=1}^{n-1}h_{k}'\right)\geq J\left( \left( \sum_{k=1}^{n-1}h_{k}'\right) ^{\ast}\right) =J(1_{B_{r}}).
	\end{align}
\end{proof}

Next, we show that $ 1_{D} $ is the unique minimizer of the variational problem \eqref{Var2}.

\begin{prop}[Uniqueness]\label{prop1-1}
	\begin{equation*}
		S'=\left\lbrace 1_{D}\right\rbrace .\label{S'=1D}
	\end{equation*}
\end{prop}

\begin{proof}
	The equation \eqref{S=1D} and Lemma \ref{lem1-1} show $ S=\left\lbrace 1_{D}\right\rbrace \subset S' $, so it suffices to show
	\begin{equation}
		\nonumber S'\subset S.
	\end{equation}
	We let $ \phi\in S' $. Due to $ I'=I $ by Lemma \ref{lem1-1}, we want
	\begin{equation*}
		\phi\in P.\label{phiP}
	\end{equation*}
	To prove this, we need to show
	\begin{equation}
		\nonumber\left| \left\lbrace x\in\mathbb{R}^{2} : 0<\phi(x)<1\right\rbrace \right| =0.
	\end{equation}
	For a contradiction, we suppose $ \left| \left\lbrace x\in\mathbb{R}^{2} : 0<\phi(x)<1\right\rbrace \right| >0. $ Then there exists $ \delta>0 $ such that
	\begin{equation}
		\left| \left\lbrace x\in\mathbb{R}^{2}:\delta\leq\phi(x)\leq1-\delta\right\rbrace \right| >0.\label{|A|>0}
	\end{equation}
	We let
	\begin{gather}
		\nonumber A=\left\lbrace x\in\mathbb{R}^{2}: \delta\leq\phi(x)\leq1-\delta\right\rbrace .
	\end{gather}
	Then we have $ 0<|A|<\infty $ and $ 0<J(1_{A})<\infty $ because $ \phi $ lies on $ P' $. To begin with, we fix a function $ h\in L^{\infty}(\mathbb{R}^{2}) $ that satisfies
	\begin{equation}
		\text{supp }(h)\subset A,\quad \int_{\mathbb{R}^{2}}hdx=1, \quad J(h)=0.\label{hprop}
	\end{equation}
	For instance, we can construct such $ h $ explicitly in the following way: First, take positive real numbers $ 0<r_{1}<r_{2}<\infty $ such that
	\begin{equation}
		\nonumber 0<|A\cap B_{r_{1}}|<\infty,\quad 0<|A\setminus B_{r_{2}}|<\infty.
	\end{equation}
	Second, take $ c_{1},c_{2}\in\mathbb{R} $ that satisfy
	\begin{gather}
		\begin{split}
			c_{1}|A\cap B_{r_{1}}|&+c_{2}|A\setminus B_{r_{2}}|=1,\\
			c_{1}J(1_{A\cap B_{r_{1}}})&+c_{2}J(1_{A\setminus B_{r_{2}}})=0.\label{c1c2}
		\end{split}
	\end{gather}
	Such $ c_{1},c_{2} $ always exist because the linear system \eqref{c1c2} has a non-singular matrix;
	\begin{align}
		\nonumber \begin{vmatrix}
			|A\cap B_{r_{1}}| & |A\setminus B_{r_{2}}|\\
			J(1_{A\cap B_{r_{1}}}) & J(1_{A\setminus B_{r_{2}}})
		\end{vmatrix}&=|A\cap B_{r_{1}}|\int_{A\setminus B_{r_{2}}}|x|^{2}dx-|A\setminus B_{r_{2}}|\int_{A\cap B_{r_{1}}}|x|^{2}dx\\
		\nonumber&\geq(r_{2}^{2}-r_{1}^{2})|A\cap B_{r_{1}}||A\setminus B_{r_{2}}|>0.
	\end{align}
	Finally, define $ h $ as
	\begin{equation}
		\nonumber h(x)=\begin{cases}
			c_{1} & \text{ if }x\in A\cap B_{r_{1}},\\
			c_{2} & \text{ if }x\in A\setminus B_{r_{2}},\\
			0 & \text{ otherwise}.
		\end{cases}
	\end{equation}
	Then we see such $ h $ satisfies the condition \eqref{hprop}.
	
	Now we define $ \eta\in L^{\infty}(\mathbb{R}^{2}) $ by
	\begin{equation}
		\nonumber\eta=1_{A}-|A|h.
	\end{equation}
	Then $ \eta $ is supported on $ A $ and it satisfies
	\begin{equation}
		\nonumber\int_{\mathbb{R}^{2}}\eta dx=|A|-|A|\int_{\mathbb{R}^{2}}hdx=0,
	\end{equation}
	\begin{equation}
		J(\eta)=J(1_{A})-|A|J(h)=J(1_{A})\in(0,\infty).\label{Jeta=J1A}
	\end{equation}	
	Now we take $ \varepsilon_{0}>0 $ small enough such that $ \varepsilon_{0}\left\|\eta \right\| _{L^{\infty}}\leq\delta $. Then we have
	\begin{equation}
		\nonumber 0\leq(\phi-\varepsilon_{0}\eta)\leq1.
	\end{equation}
	In addition, we also have
	\begin{equation}
		\nonumber J(\phi-\varepsilon_{0}\eta)=J(\phi)-\varepsilon_{0} J(\eta)<\infty,\quad\int_{\mathbb{R}^{2}}(\phi-\varepsilon_{0}\eta)dx=\int_{\mathbb{R}^{2}}\phi dx=\pi,
	\end{equation}
	which shows $ (\phi-\varepsilon_{0}\eta)\in P' $. Then due to \eqref{Jeta=J1A} and the initial assumption $ \phi\in S' $, we have
	\begin{align}
		\nonumber J(\phi)\leq J(\phi-\varepsilon_{0}\eta)=J(\phi)-\varepsilon_{0} J(\eta)=J(\phi)-\varepsilon_{0} J(1_{A}).
	\end{align}
	This implies
	\begin{equation}
		\nonumber -\varepsilon_{0}J(1_{A})\geq0,
	\end{equation}
	which gives us
	$
		J(1_{A})=0.
	$
	Thus, we have $ |A|=0 $, which is a contradiction to \eqref{|A|>0}.	
\end{proof}

\subsection{Compactness}

In this subsection, we prove that a sequence of functions which has norm convergences with the convergence
$$ J(f_{n})\longrightarrow J(1_{D})$$
 should have strong convergences.

\begin{prop}\label{prop2}
	Let $ \left\lbrace f_{n}\right\rbrace_{n=1}^{\infty} \subset L^{\infty}(\mathbb{R}^{2}) $ be a sequence of nonnegative functions such that
	\begin{gather}
		\int_{\left\lbrace x\in\mathbb{R}^{2}: f_{n}(x)\geq1+a_{n}\right\rbrace }f_{n}dx\longrightarrow0\quad\text{ for some sequence }\left\lbrace a_{n}\right\rbrace_{n=1}^{\infty}\subset\mathbb{R}_{>0} \text{ satisfying }a_{n}\searrow0,\label{fn<=1}\\
		\left\| f_{n}\right\| _{L^{1}}\longrightarrow\left\| 1_{D}\right\| _{L^{1}},\label{L1conv}\\
		\left\| f_{n}\right\| _{L^{2}}\longrightarrow\left\| 1_{D}\right\| _{L^{2}},\label{L2conv}\\
		J(f_{n})\longrightarrow J(1_{D})\quad\text{ as }\quad n\longrightarrow\infty.\label{Jconv}
	\end{gather}
	Then $ \left\lbrace f_{n}\right\rbrace  $ satisfies
	\begin{gather}
		\nonumber \left\| f_{n}- 1_{D}\right\| _{L^{1}}+\left\| f_{n}- 1_{D}\right\| _{L^{2}}+J(|f_{n}- 1_{D}|) \longrightarrow0\quad\text{ as }\quad n\longrightarrow\infty.
	\end{gather}
\end{prop}

\begin{proof}
	Suppose that the conclusion is false, that is, there exists $ \varepsilon_{0}>0 $ and a subsequence $ \left\lbrace f_{n_{j}}\right\rbrace_{j=1}^{\infty} $ of $ \left\lbrace f_{n}\right\rbrace _{n=1}^{\infty} $ such that
	\begin{equation}
		\left\| f_{n_{j}}- 1_{D}\right\| _{L^{1}}+\left\| f_{n_{j}}- 1_{D}\right\| _{L^{2}}+J\left( \left| f_{n_{j}}- 1_{D}\right| \right) \geq\varepsilon_{0}\quad\text{ for all }\quad j\in\mathbb{N}.\label{cont}
	\end{equation}
	Now we show that there exists a subsequence $ \left\lbrace f_{n_{j_{k}}}\right\rbrace_{k=1}^{\infty} $ of $ \left\lbrace f_{n_{j}}\right\rbrace_{j=1}^{\infty} $ such that
	\begin{equation*}
		\left\| f_{n_{j_{k}}}- 1_{D}\right\| _{L^{1}}+\left\| f_{n_{j_{k}}}- 1_{D}\right\| _{L^{2}}+J\left( \left| f_{n_{j_{k}}}- 1_{D}\right| \right) \longrightarrow0\quad\text{ as }\quad k\longrightarrow\infty.\label{fnjk}
	\end{equation*}
	Once this is shown, then it contradicts \eqref{cont}, completing the proof.
	
	Due to
	\begin{equation}
		\nonumber\sup_{j\in\mathbb{N}}\left\| f_{n_{j}}\right\| _{L^{2}}<\infty,
	\end{equation}
	which comes from the $ L^{2} $-norm convergence \eqref{L2conv}, there exists $ f\in L^{2}(\mathbb{R}^{2}) $ and a subsequence $ \left\lbrace f_{n_{j_{k}}}\right\rbrace _{k=1}^{\infty} $ of $ \left\lbrace f_{n_{j}}\right\rbrace _{j=1}^{\infty} $ such that
	\begin{equation}
		f_{n_{j_{k}}}\rightharpoonup f\quad\text{ in }\quad L^{2}(\mathbb{R}^{2})\quad\text{ as }\quad k\longrightarrow\infty.\label{weakL2}
	\end{equation}
	For notational convenience, we simply denote the subsequence $ \left\lbrace f_{n_{j_{k}}}\right\rbrace _{k=1}^{\infty} $ by $ \left\lbrace f_{n}\right\rbrace  $. The proof is done in 8 steps:\\
	
	\textbf{Step 1.} $ 0\leq f\leq1 $.
	
	\textbf{Step 2.} $ \left\| f\right\| _{L^{1}}\leq\left\| 1_{D}\right\| _{L^{1}}\text{ and } J(f)\leq J(1_{D}) $.
	
	\textbf{Step 3.} $ \left\| f\right\| _{L^{1}}=\left\| 1_{D}\right\| _{L^{1}}. $
	
	\textbf{Step 4.} $ J(f)=J(1_{D}) $.
	
	\textbf{Step 5.} $ f=1_{D} $.
	
	\textbf{Step 6.} $ f_{n}\longrightarrow 1_{D} $ in $ L^{2} $ as $ n\longrightarrow\infty $.
	
	\textbf{Step 7.} $ J(|f_{n}-1_{D}|)\longrightarrow0 $ as $ n\longrightarrow\infty $.
	
	\textbf{Step 8.} $ f_{n}\longrightarrow 1_{D} $ in $ L^{1} $ as $ n\longrightarrow\infty $.\\
	
	Step 1, 2, and 3 are to show $ f\in P' $. Then Step 4 and Proposition \ref{prop1-1} show Step 5. This step confirms that the $ L^{2} $-weak limit $ f $ is indeed the unique minimizer $ 1_{D} $ in $ P' $. The remaining steps finish our proof.\\
	
	\textbf{Step 1.} $ 0\leq f\leq1 $.\\
	
	We trivially have $ f\geq0 $ due to $ f_{n}\geq0 $ for each $ n\in\mathbb{N} $. To prove $ f\leq1 $, we show
	\begin{equation}
		\nonumber\left| \left\lbrace x\in\mathbb{R}^{2} : f(x)>1\right\rbrace \right| =0.\label{f>1=0}
	\end{equation}
	For a contradiction, we suppose
	\begin{equation}
		\nonumber\left| \left\lbrace x\in\mathbb{R}^{2} : f(x)>1\right\rbrace\right| >0.\label{fx>1>0}
	\end{equation}
	Then there exists $ \delta>0 $ such that
	\begin{equation}
		0<\left| \left\lbrace x\in\mathbb{R}^{2} : f(x)\geq1+\delta\right\rbrace \right|<\infty,\label{1+delta}
	\end{equation}
	where the measure is finite because of $ f\in L^{2}(\mathbb{R}^{2}) $. We denote
	\begin{equation}
		\nonumber C=\left\lbrace x\in\mathbb{R}^{2} : f(x)\geq1+\delta\right\rbrace,\quad E_{n}=\left\lbrace x\in\mathbb{R}^{2} :  f_{n}(x)\geq1+a_{n}\right\rbrace.
	\end{equation}
	Then we have
	\begin{equation*}
		\int_{C} f dx\geq\int_{C} \left( 1+\delta\right) dx=\left( 1+\delta\right) \left|C\right| ,\label{fdx}
	\end{equation*}
	and
	\begin{equation*}
		\begin{split}
			\int_{C}f_{n}dx&=\int_{C\cap E_{n}}f_{n}dx+\int_{C\setminus E_{n}}f_{n}dx\leq\int_{E_{n}}f_{n}dx+\int_{C\setminus E_{n}}\left( 1+ a_{n}\right) dx\leq\int_{E_{n}}f_{n}dx+\left( 1+ a_{n}\right)\left|C\right|.\label{fndx}
		\end{split}		
	\end{equation*}
	Then we obtain
	\begin{align}
		\int_{C}\left( f_{n}-f\right) dx&\leq\int_{E_{n}}f_{n}dx+\left( a_{n}-\delta\right)\left|C\right|.\label{fn-f}
	\end{align}
	Then the left-hand side of \eqref{fn-f} goes to $ 0 $ because of the $ L^{2}$-weak convergence \eqref{weakL2} and the finite-measured set \eqref{1+delta}, and the right-hand side of \eqref{fn-f} converges to $ -\delta\left|C\right| $ due to \eqref{fn<=1} as $ n\longrightarrow\infty $. This gives 
	 a contradiction to \eqref{1+delta}.\\
	
	\textbf{Step 2.} $ \left\| f\right\| _{L^{1}}\leq\left\| 1_{D}\right\| _{L^{1}}, J(f)\leq J(1_{D}) $.\\
	
	We observe
	\begin{equation}
		\int_{\mathbb{R}^{2}}f_{n}1_{B_{r}}dx=\int_{B_{r}}f_{n}dx\leq\left\| f_{n}\right\| _{L^{1}},\quad r>0.\label{Step2-1}
	\end{equation}
	Then because of the $ L^{1} $-norm convergence \eqref{L1conv} and the $ L^{2} $-weak convergence \eqref{weakL2}, we have by taking $ n\longrightarrow\infty $ on both sides of \eqref{Step2-1},
	\begin{equation*}
		\int_{B_{r}}fdx\leq\left\| 1_{D}\right\| _{L^{1}},\quad r>0.\label{Step2-2}
	\end{equation*}
	Then we have
	\begin{align}
		\nonumber \left\| f\right\| _{L^{1}}&=\lim\limits_{r\rightarrow\infty}\int_{\mathbb{R}^{2}}f1_{B_{r}}dx=\lim\limits_{r\rightarrow\infty}\int_{B_{r}}fdx\leq\left\| 1_{D}\right\| _{L^{1}}.
	\end{align}	
	
	Similarly, by replacing $ 1_{B_{r}} $ with $ |x|^{2}1_{B_{r}}\in L^{2}(\mathbb{R}^2) $, the convergence \eqref{Jconv}  of angular impulse and the $ L^{2} $-weak convergence \eqref{weakL2} produce
	\begin{align}
		J(f)\leq J(1_{D}).\label{Jf<=J1D}
	\end{align}
	
	\textbf{Step 3.} $ \left\| f\right\| _{L^{1}}=\left\| 1_{D}\right\| _{L^{1}}. $\\
	
	To show this, we fix $ \varepsilon>0 $. Then because of $ f\in L^{1}(\mathbb{R}^{2}) $, there exists $ R_{1}>0 $ such that
	\begin{equation}
		\int_{\mathbb{R}^{2}\setminus B_{R}}fdx\leq\varepsilon\quad\text{ for all }\quad R\geq R_{1}.\label{eps1}
	\end{equation}
	Furthermore, by the convergence  \eqref{Jconv} of angular impulse, we have
	\begin{equation}
		\nonumber \sup_{n\in\mathbb{N}}J(f_{n})<\infty.
	\end{equation}
	Let's denote $ M=\sup_{n\in\mathbb{N}}J(f_{n}) $. Then due to
	\begin{gather*}
		M\geq J(f_{n})\geq\int_{\mathbb{R}^{2}\setminus B_{r}}|x|^{2}f_{n}dx\geq r^{2}\int_{\mathbb{R}^{2}\setminus B_{r}}f_{n}dx\quad\text{ for all }\quad n\in\mathbb{N}\quad\text{ and }\quad r>0,\label{M>=Jfn}
	\end{gather*}
	we get
	\begin{equation}
		\nonumber\frac{M}{r^{2}}\geq\sup_{n\in\mathbb{N}}\int_{\mathbb{R}^{2}\setminus B_{r}}f_{n}dx\quad\text{ for all }\quad r>0.
	\end{equation}
	So choosing $ R_{2}\geq\sqrt{\frac{M}{\varepsilon}} $, we have
	\begin{equation}
		\sup_{n\in\mathbb{N}}\int_{\mathbb{R}^{2}\setminus B_{R}}f_{n}dx\leq\varepsilon\quad\text{ for all }\quad R\geq R_{2}.\label{eps2}
	\end{equation}
	We take $ R=\max\left\lbrace R_{1},R_{2}\right\rbrace  $. Then because of $ 1_{B_{R}}\in L^{2}(\mathbb{R}^{2}) $ and the $ L^{2} $-weak convergence \eqref{weakL2}, there exists $ N\in\mathbb{N} $ such that
	\begin{equation}
		\left| \int_{B_{R}}f_{n}dx-\int_{B_{R}}fdx\right| \leq\varepsilon\quad\text{ for all }\quad n\geq N.\label{eps3}
	\end{equation}
	By collecting all the estimates \eqref{eps1}, \eqref{eps2}, and \eqref{eps3}, we have, for all $ n\geq N $, 
	\begin{align}
		\nonumber \Bigl| \left\| f_{n}\right\| _{L^{1}}-\left\| f\right\| _{L^{1}}\Bigr| &\leq\left| \int_{B_{R}}f_{n}dx-\int_{B_{R}}fdx\right| +\left| \int_{\mathbb{R}^{2}\setminus B_{R}}f_{n}dx-\int_{\mathbb{R}^{2}\setminus B_{R}}fdx\right| \\
		\nonumber &\leq\left| \int_{B_{R}}f_{n}dx-\int_{B_{R}}fdx\right| +\int_{\mathbb{R}^{2}\setminus B_{R}}f_{n}dx+\int_{\mathbb{R}^{2}\setminus B_{R}}fdx\\
		\nonumber&\leq\varepsilon+\varepsilon+\varepsilon=3\varepsilon.
	\end{align}
	Hence by the $ L^{1} $-norm convergence \eqref{L1conv}, we have
	\begin{equation}
		\nonumber \left\| f\right\| _{L^{1}}=\lim\limits_{n\rightarrow\infty}\left\| f_{n}\right\| _{L^{1}}=\left\| 1_{D}\right\| _{L^{1}}.
	\end{equation}
	
	\textbf{Step 4.} $ J(f)=J(1_{D}) $.\\
	
	Step 1, 2, and 3 show
	\begin{equation}
		f\in P'.\label{finP'}
	\end{equation}
	Then by Proposition \ref{prop1-1}, we have
	\begin{equation}
		\nonumber J(f)\geq I'=J(1_{D}).
	\end{equation}
	Then by the inequality \eqref{Jf<=J1D} in Step 2, we obtain
	\begin{equation}
		J(f)=J(1_{D}).\label{Jf=J1D}
	\end{equation}
	
	\textbf{Step 5.} $ f=1_{D} $.\\
	
	By \eqref{finP'} and by \eqref{Jf=J1D},   Proposition \ref{prop1-1} implies 
	\begin{equation}
		f\in S'=\left\lbrace 1_{D}\right\rbrace .\label{f=1D}
	\end{equation}	
	Up to now, we have shown that the $ L^{2} $-weak limit $ f $ is exactly $ 1_{D} $. From now on, we show that the weak limit is, in fact, the strong limit.\\
	
	\textbf{Step 6.} $ f_{n}\longrightarrow 1_{D} $ in $ L^{2} $ as $ n\longrightarrow\infty $.\\
	
	From the $ L^{2} $-weak convergence \eqref{weakL2} and \eqref{f=1D} in Step 5, we have
	$$ f_{n}\rightharpoonup 1_{D} \quad\text{ in }\quad L^{2}\quad\text{ as }\quad n\longrightarrow\infty. $$
	Then this together with the $ L^{2} $-norm convergence \eqref{L2conv} gives us
	\begin{equation}
		f_{n}\longrightarrow 1_{D} \quad\text{ in }\quad L^{2}\quad \text{ as }\quad n\longrightarrow\infty.\label{strongL2}
	\end{equation}
	
	\textbf{Step 7.} $ J(|f_{n}-1_{D}|)\longrightarrow0 $ as $ n\longrightarrow\infty $.\\
	
	To show this, we split the range of the integral of $ J(|f_{n}-1_{D}|) $ into $ D $ and $ \mathbb{R}^{2}\setminus D $:
	\begin{align}
		\nonumber J(|f_{n}-1_{D}|)&=\underbrace{\int_{D}|x|^{2}|f_{n}-1_{D}|dx}_{(I)}+\underbrace{\int_{\mathbb{R}^{2}\setminus D}|x|^{2}f_{n}dx}_{(II)}.
	\end{align}
	For $ (I) $, using the H\"{o}lder's inequality, by the strong $ L^{2} $-convergence \eqref{strongL2} in Step 6, we have
	\begin{align}
		\nonumber (I)&\leq\left( \int_{\mathbb{R}^{2}}|f_{n}-1_{D}|^{2}dx\right) ^{\frac{1}{2}}\left( \int_{\mathbb{R}^{2}}|x|^{4}|1_{D}|^{2}dx\right) ^{\frac{1}{2}}\leq C\left\| f_{n}-1_{D}\right\| _{L^{2}}\longrightarrow0\quad\text{ as }\quad n\longrightarrow\infty.
	\end{align}
	To show the same holds for $ (II) $, we use the following decomposition;
	\begin{align}
		\nonumber \int_{\mathbb{R}^{2}\setminus D}|x|^{2}f_{n}dx&=\Big(J(f_{n})-J(1_{D})\Big)-\left( \int_{D}|x|^{2}f_{n}dx-\int_{D}|x|^{2}1_{D}dx\right) .
	\end{align}
	Again using the H\"{o}lder's inequality, by the convergence \eqref{Jconv} of angular impulse  and by \eqref{strongL2}, we have
	\begin{align}
		\nonumber(II)&\leq \left| J(f_{n})-J(1_{D})\right|+\left|  \int_{D}|x|^{2}f_{n}dx-\int_{D}|x|^{2}1_{D}dx\right|\\
		\nonumber&\leq\left| J(f_{n})-J(1_{D})\right|+C\left\| f_{n}-1_{D}\right\| _{L^{2}}\longrightarrow0\quad\text{ as }\quad n\longrightarrow\infty.
	\end{align}
	Therefore, we have $ J(|f_{n}-1_{D}|)\longrightarrow0 $ as $ n\longrightarrow\infty $.\\
	
	\textbf{Step 8.} $ f_{n}\longrightarrow 1_{D} $ in $ L^{1} $ as $ n\longrightarrow\infty $.\\
	
	By the H\"{o}lder's inequality, the $ L^{1} $-norm can be controlled by the $ {J}_{2} $-norm as the following:
	\begin{equation}
		\left\|g\right\| _{L^{1}}=\int_{D}|g|dx+\int_{\mathbb{R}^{2}\setminus D}|g|dx\leq\sqrt{\pi}\left\|g\right\| _{L^{2}}+J(|g|)\leq\pi\left\| g\right\| _{{J}_{2}},\quad g\in L^{2},\quad J(|g|)<\infty.\label{L1L2Lw1}
	\end{equation}
	Then by Step 6, 7, and the above with $ g=f_{n}-1_{D} $, we have $ \left\| f_{n}-1_{D}\right\| _{L^{1}}\longrightarrow0 $ as $ n\longrightarrow\infty $.
\end{proof}

\subsection{Contradiction argument}

\begin{proof}[Proof of Theorem \ref{thm1}]
	
	We suppose that the conclusion is false. Then there exists $ \varepsilon_{0}>0 $ such that there exists a sequence of nonnegative initial data $ \left\lbrace \omega_{n,0} \right\rbrace_{n=1}^{\infty}\subset L^{\infty}(\mathbb{R}^{2}) $ with $ J(\omega_{n,0})<\infty $ and a sequence $ \left\lbrace t_{n}\right\rbrace _{n=1}^{\infty}\subset\mathbb{R}_{\geq0} $ such that for each $ n\in\mathbb{N} $, we have
	\begin{equation}
		\left\| \omega_{n,0}-1_{D}\right\| _{{J}_{2}}\leq\frac{1}{n},\label{eq8}
	\end{equation}
	but
	\begin{equation}
		\left\| \omega_{n}(t_{n})-1_{D}\right\| _{{J}_{2}}\geq\varepsilon_{0},\label{eq9}
	\end{equation}
	where $ \omega_{n}(t) $ is the solution of \eqref{eq1} for $ \omega_{n,0} $. First, we recall that for each $ n $, the $ L^{p} $-norm for $ p\in[1,\infty] $ and $ J\left( \omega_{n}(t)\right)  $ are preserved in time. This gives us
	\begin{equation}
		\left\| \omega_{n}(t_{n})\right\| _{L^{1}}=\left\| \omega_{n,0}\right\| _{L^{1}},\quad \left\| \omega_{n}(t_{n})\right\| _{L^{2}}=\left\| \omega_{n,0}\right\| _{L^{2}},\quad J(\omega_{n}(t_{n}))=J(\omega_{n,0}).\label{conserved}
	\end{equation}
	Also we recall that for each $ n $, the corresponding flow for $ \omega_{n}(t) $ preserves the measure of each level set in time. Thus, $ \omega_{n}(t_{n}) $ is nonnegative and we get
	\begin{equation}
		\int_{\left\lbrace \omega_{n}(t_{n})\geq1+\frac{1}{\sqrt{n}}\right\rbrace }\omega_{n}(t_{n})dx=\int_{\left\lbrace \omega_{n,0}\geq1+\frac{1}{\sqrt{n}}\right\rbrace }\omega_{n,0}dx.\label{eq12}
	\end{equation}
	For notational convenience, we denote $ \omega_{n}=\omega_{n}(t_{n}) $. By the estimate \eqref{L1L2Lw1} and \eqref{eq8} above, we have
	\begin{equation}
		\left\| \omega_{n,0}-1_{D}\right\|_{L^{1}}\leq\pi\left\| \omega_{n,0}-1_{D}\right\|_{{J}_{2}}\leq\frac{\pi}{n}.\label{L1bound}
	\end{equation}
	Then by \eqref{eq8}, \eqref{conserved}, and \eqref{L1bound}, we get
	\begin{equation}
		\begin{split}
			\left\| \omega_{n}\right\| _{L^{1}}&=\left\| \omega_{n,0}\right\| _{L^{1}}\longrightarrow\left\| 1_{D}\right\| _{L^{1}},\\
			\left\| \omega_{n}\right\| _{L^{2}}&=\left\| \omega_{n,0}\right\| _{L^{2}}\longrightarrow\left\| 1_{D}\right\| _{L^{2}},\\
			J(\omega_{n})&=J(\omega_{n,0})\longrightarrow J(1_{D})\quad\text{ as }\quad n\longrightarrow\infty.
		\end{split}\label{condwn}		
	\end{equation}
	We claim
	\begin{equation}
		\int_{\left\lbrace \omega_{n,0}\geq1+\frac{1}{\sqrt{n}}\right\rbrace }\omega_{n,0}dx\longrightarrow0\quad\text{ as }\quad n\longrightarrow\infty.\label{eq11}
	\end{equation}
	Indeed, we split the range of the above integral;
	\begin{equation}
		\nonumber \int_{\left\lbrace \omega_{n,0}\geq1+\frac{1}{\sqrt{n}}\right\rbrace }\omega_{n,0}dx=\underbrace{\int_{\left\lbrace \omega_{n,0}\geq1+\frac{1}{\sqrt{n}}\right\rbrace \cap D}\omega_{n,0}dx}_{(I)}+\underbrace{\int_{\left\lbrace \omega_{n,0}\geq1+\frac{1}{\sqrt{n}}\right\rbrace \setminus D}\omega_{n,0}dx}_{(II)}.
	\end{equation}
	For $ (II) $, by the $ L^{1} $-convergence \eqref{L1bound}, we have
	\begin{gather}
		\nonumber (II)\leq\int_{\mathbb{R}^{2}\setminus D}\omega_{n,0}dx\leq\left\| \omega_{n,0}-1_{D}\right\|_{L^{1}}\leq\frac{\pi}{n}\longrightarrow0\quad\text{ as }\quad n\longrightarrow\infty.
	\end{gather}
	To prove the similar holds for $ (I) $, we use the $ L^{1} $-convergence \eqref{L1bound} once more;
	\begin{align}
		\nonumber \frac{\pi}{n}&\geq\left\| \omega_{n,0}-1_{D}\right\|_{L^{1}}\geq\int_{D}|\omega_{n,0}-1|dx\\
		\nonumber&\geq\int_{\left\lbrace \omega_{n,0}\geq1+\frac{1}{\sqrt{n}}\right\rbrace \cap D}|\omega_{n,0}-1|dx\geq\frac{1}{\sqrt{n}}\left| \left\lbrace \omega_{n,0}\geq1+\frac{1}{\sqrt{n}}\right\rbrace \cap D\right|,
	\end{align}
	which gives us
	\begin{equation}
		\nonumber\left| \left\lbrace \omega_{n,0}\geq1+\frac{1}{\sqrt{n}}\right\rbrace \cap D\right|\leq\frac{\pi}{\sqrt{n}}.
	\end{equation}
	Then we have
	\begin{equation}
		\begin{split}
			\nonumber(I)&=\int_{\left\lbrace \omega_{n,0}\geq1+\frac{1}{\sqrt{n}}\right\rbrace \cap D}(\omega_{n,0}-1_{D})dx+\int_{\left\lbrace \omega_{n,0}\geq1+\frac{1}{\sqrt{n}}\right\rbrace \cap D}1_{D}dx\\
			&\leq\left\| \omega_{n,0}-1_{D}\right\| _{L^{1}}+\left| \left\lbrace \omega_{n,0}\geq1+\frac{1}{\sqrt{n}}\right\rbrace \cap D\right|\longrightarrow0\quad\text{ as }\quad n\longrightarrow\infty.
		\end{split}
	\end{equation}
	Hence, \eqref{eq11} holds. Then by the equation \eqref{eq12}, we have
	\begin{equation}
		\int_{\left\lbrace \omega_{n}\geq1+\frac{1}{\sqrt{n}}\right\rbrace }\omega_{n}dx\longrightarrow0\quad\text{ as }\quad n\longrightarrow\infty.\label{wndx}
	\end{equation}
	In sum, we obtained all the assumptions of Proposition \ref{prop2} from \eqref{condwn} and \eqref{wndx}. Therefore by the proposition, the sequence $ \left\lbrace \omega_{n}\right\rbrace _{n=1}^{\infty} $ satisfies
	\begin{equation}
		\nonumber \left\| \omega_{n}-1_{D}\right\| _{{J}_{2}}\longrightarrow0\quad\text{ as }\quad n\longrightarrow\infty,
	\end{equation}
	which is a contradiction to \eqref{eq9}.
	
\end{proof}

\section{Stability estimate}

\subsection{Rearrangement estimates}

Throughout this section, we use the following fine properties.

\begin{lem}[Nonexpansivity of rearrangement]\label{lemnonexp}
	For nonnegative functions $ g,h\in L^{1}(\mathbb{R}^{2}) $,
	we have
	\begin{equation}
		\left\| g^{\ast}-h^{\ast}\right\| _{L^{1}}\leq\left\| g-h\right\| _{L^{1}}.\label{nonexp}
	\end{equation}
\end{lem}
The result follows from the convexity of $ |\cdot| $. For a proof, one may refer \cite[Sec. 3.5]{LiLAn} for details. The reason why we need $ \zeta $ to be radially symmetric and non-increasing in order to get stability is because it satisfies $ \zeta^{\ast}=\zeta $. Additionally, we shall use the rearrangement estimate \eqref{keyestimate} below, which is a sharper version of the basic property \eqref{Jbasic}. We prove the estimate by following the spirit of \cite[Lemma 1]{MaP85}. As our setting is the whole space $ \mathbb{R}^{2} $ while the lemma of \cite{MaP85} was stated for Euler flows in bounded channel $ \mathbb{T}\times[0,R] $, $ R>0$, we present the proof below for completeness.

\begin{lem}\label{lem2}
	Let a nonnegative function $ f\in L^{\infty}(\mathbb{R}^{2}) $ satisfy $ J(f)<\infty $ and $ f^{\ast} $ be its rearrangement. Then they satisfy
	\begin{equation}
		\left\| f-f^{\ast}\right\| _{L^{1}}^{2}\leq4\pi\left\| f\right\| _{L^{\infty}}\bigg[ J(f)-J(f^{\ast}) \bigg].\label{keyestimate}
	\end{equation}
\end{lem}

\begin{proof}
	Due to
	\begin{equation}
		\nonumber J(af)=aJ(f)\quad\text{ and }\quad(af)^{\ast}=af^{\ast},\quad a>0,
	\end{equation}
	it's enough to show only for the case $ \left\| f\right\| _{L^{\infty}}=1 $. Then for each $ n\in\mathbb{N} $, we use similar level sets and simple functions from \eqref{FkGk} and \eqref{fngn} in the proof of Lemma \ref{lem1-1};
	\begin{equation}
		\nonumber A_{k}^{(n)}=\left\lbrace x\in\mathbb{R}^{2} : f(x)>\frac{k}{n}\right\rbrace ,\quad C_{k}^{(n)}=\left\lbrace x\in\mathbb{R}^{2} : f^{\ast}(x)>\frac{k}{n}\right\rbrace,\quad k=1,\cdots,n-1,
	\end{equation}
	\begin{equation}
		\nonumber \xi^{(n)}=\sum_{k=1}^{n-1}\frac{1}{n}1_{A^{(n)}_{k}},\quad \eta^{(n)}=\sum_{k=1}^{n-1}\frac{1}{n}1_{C^{(n)}_{k}},
	\end{equation}
	in which $ \xi^{(n)} $ and $ \eta^{(n)} $ are dominated by and converge to $ f $ and $ f^{\ast} $ pointwise, respectively. We claim
	\begin{equation}
		\frac{4\pi (n-1)}{n}\Big[J(\xi^{(n)})-J(\eta^{(n)})\Big]\geq\left\| \xi^{(n)}-\eta^{(n)}\right\| _{L^{1}}^{2},\quad n\in\mathbb{N}.\label{eqn1}
	\end{equation}
	Once we have this, by the dominated convergence theorem, taking $ n\longrightarrow\infty $ on both sides of the above claim gives us \eqref{keyestimate}.
	
	Let's prove the above claim. We fix $ n\in\mathbb{N} $ and for the convenience of notation, we drop the parameter $ n $;
	\begin{equation}
		\nonumber A_{k}=A_{k}^{(n)},\quad C_{k}=C_{k}^{(n)},\quad\xi=\xi^{(n)},\quad \eta=\eta^{(n)}.
	\end{equation}
	Additionally, we take $ s_{k}, \beta_{k}\geq0 $ satisfying
	\begin{equation}
		\nonumber C_{k}=B_{s_{k}},\quad \beta_{k}=\frac{1}{2}\left\| 1_{A_{k}}-1_{C_{k}}\right\| _{L^{1}}=\frac{1}{2}|A_{k}\bigtriangleup C_{k}|,\quad k=1,\cdots n-1.
	\end{equation}		
	We recall $ |A_{k}|=|C_{k}| $ for each $ k $, due to the property \eqref{levelsetmeas}. This gives us
	\begin{equation}
		|A_{k}\setminus C_{k}|=|C_{k}\setminus A_{k}|=\frac{1}{2}|A_{k}\bigtriangleup C_{k}|=\beta_{k},\quad k=1,\cdots n-1.\label{betak}
	\end{equation}	
	Now we compute
	\begin{align}
		\nonumber J(\xi)-J(\eta)&=\frac{1}{n}\sum_{k=1}^{n-1}\left(\int_{A_{k}}|x|^{2}dx-\int_{C_{k}}|x|^{2}dx \right)=\frac{1}{n}\sum_{k=1}^{n-1}\left(\int_{A_{k}\setminus C_{k}}|x|^{2}dx-\int_{C_{k}\setminus A_{k}}|x|^{2}dx \right)\\
		\nonumber&\geq\frac{1}{n}\sum_{k=1}^{n-1}\left(\int_{\Sigma_{k,1}}|x|^{2}dx-\int_{\Sigma_{k,2}}|x|^{2}dx \right),
	\end{align}
	where for each $ k $, $ \Sigma_{k,1}\subset\mathbb{R}^{2}\setminus B_{s_{k}} $ is the annulus that has the same measure as $ A_{k}\setminus C_{k} $ with the inner radius $ s_{k} $, and $ \Sigma_{k,2}\subset B_{s_{k}} $ is the annulus having the same measure as $ C_{k}\setminus A_{k} $ with the outer radius $ s_{k} $;
	\begin{equation}
		\nonumber\Sigma_{k,1}=\left\lbrace x\in\mathbb{R}^{2} : s_{k}\leq|x|\leq r_{k,1}\right\rbrace  ,\quad \Sigma_{k,2}=\left\lbrace x\in\mathbb{R}^{2} : r_{k,2}\leq|x|\leq s_{k}\right\rbrace ,
	\end{equation}
	with $ r_{k,1}\geq s_{k}\geq r_{k,2}\geq0 $ satisfying
	\begin{equation}
		\nonumber|\Sigma_{k,1}|=|A_{k}\setminus C_{k}|,\quad |\Sigma_{k,2}|=|C_{k}\setminus A_{k}|.
	\end{equation}
	Thus, by \eqref{betak}, we have
	\begin{equation}
		\nonumber |\Sigma_{k,1}|=|\Sigma_{k,2}|=\beta_{k},
	\end{equation}
	in which a further calculation gives us
	\begin{equation}
		\nonumber r_{k,1}=\sqrt{s_{k}^{2}+\frac{\beta_{k}}{\pi}},\quad r_{k,2}=\sqrt{s_{k}^{2}-\frac{\beta_{k}}{\pi}}.
	\end{equation}
	Then we have
	\begin{align}
		\nonumber J(\xi)-J(\eta)&\geq\frac{1}{n}\sum_{k=1}^{n-1}\left(\int_{\Sigma_{k,1}}|x|^{2}dx-\int_{\Sigma_{k,2}}|x|^{2}dx \right)=\frac{2\pi}{n}\sum_{k=1}^{n-1}\left(\int_{s_{k}}^{r_{k,1}}\rho^{3}d\rho-\int_{r_{k,2}}^{s_{k}}\rho^{3}d\rho \right)\\
		\nonumber&=\frac{\pi}{2n}\sum_{k=1}^{n-1}(r_{k,1}^{4}+r_{k,2}^{4}-2s_{k}^{4})=\frac{\pi}{2n}\sum_{k=1}^{n-1}\left[ \left( s_{k}^{2}+\frac{\beta_{k}}{\pi}\right) ^{2}+\left( s_{k}^{2}-\frac{\beta_{k}}{\pi}\right) ^{2}-2s_{k}^{4}\right] \\
		&=\frac{1}{\pi n}\sum_{k=1}^{n-1}\beta_{k}^{2}=\frac{1}{4\pi n}\sum_{k=1}^{n-1}\left\| 1_{A_{k}}-1_{C_{k}}\right\| _{L^{1}}^{2}.\label{Jxi-Jeta}
	\end{align}
	On the other hand, using the Cauchy-Schwartz inequality, we have
	\begin{align}
		\nonumber \left\| \xi-\eta\right\| _{L^{1}}&=\int_{\mathbb{R}^{2}}\left| \frac{1}{n}\sum_{k=1}^{n-1}\left( 1_{A_{k}}-1_{C_{k}}\right) \right|dx \\
		\nonumber&\leq\frac{1}{n}\sum_{k=1}^{n-1}\left\| 1_{A_{k}}-1_{C_{k}}\right\| _{L^{1}}\leq\frac{\sqrt{n-1}}{n}\left( \sum_{k=1}^{n-1}\left\| 1_{A_{k}}-1_{C_{k}}\right\| _{L^{1}}^{2}\right) ^{\frac{1}{2}},
	\end{align}
	which gives
	\begin{equation}
		\sum_{k=1}^{n-1}\left\| 1_{A_{k}}-1_{C_{k}}\right\| _{L^{1}}^{2}\geq\frac{n^{2}}{n-1}\left\| \xi-\eta\right\| _{L^{1}}^{2}.\label{1Ak-1Ck}
	\end{equation}
	Collecting \eqref{Jxi-Jeta} and \eqref{1Ak-1Ck}, we have \eqref{eqn1}.
\end{proof}

\subsection{Stability in $ L^{1} $}

We first prove $ L^{1} $-stability. We note that the estimate below does not depend on $ \left\| \omega_{0}\right\| _{L^{\infty}} $.

\begin{lem}\label{lem3}
	Let $ R,M,\alpha\in(0,\infty) $. Then there exist constants $ C_{1}=C_{1}(R,M,\alpha)>0 $ and $ C_{2}=C_{2}(M)>0 $ such that if $ \zeta\in L^{\infty}(\mathbb{R}^{2}) $ is nonnegative, radially symmetric, non-increasing with $ J(\zeta)<\infty $, and satisfies
	\begin{equation}
		\left\| \zeta\right\| _{L^{\infty}}\leq M,\quad \left\| \zeta\right\| _{L^{1}}\leq\alpha,\label{Malpha}
	\end{equation}
	then for any nonnegative $ \omega_{0}\in L^{\infty}(\mathbb{R}^{2}) $ with $ J(\omega_{0})<\infty $, the corresponding solution $ \omega(t) $ of \eqref{eq1} satisfies
	\begin{equation}
		\begin{split}
			\sup_{t\geq0}\left\| \omega(t)-\zeta\right\| _{L^{1}}\leq &\;C_{1}\Bigl[ \left\| \omega_{0}-\zeta\right\| _{L^{1}}^{\frac{1}{2}}+\left\| \omega_{0}-\zeta\right\| _{L^{1}}+J(|\omega_{0}-\zeta|)^{\frac{1}{2}}\Bigr]+C_{2}\left( \int_{\mathbb{R}^{2}\setminus B_{R}}|x|^{2}\zeta dx\right) ^{\frac{1}{2}}.\label{L1}
		\end{split}		
	\end{equation}
\end{lem}

\begin{proof}
	During the proof, for any nonnegative function $ f\in L^{\infty}(\mathbb{R}^{2})  $, we define a set $ A_{f}\subset\mathbb{R}^{2} $ by
	\begin{equation}
		A_{f}:=\left\lbrace x\in\mathbb{R}^{2} : f(x)\leq M+1\right\rbrace,\label{Af}
	\end{equation}
	and a function $ \widetilde{f}\in L^{\infty}(\mathbb{R}^{2})  $ by
	\begin{equation}
		\widetilde{f}(x)=\begin{cases}
			f(x) & \text{ if }x\in A_{f}\\
			M+1 & \text{ if }x\in \mathbb{R}^{2}\setminus A_{f}
		\end{cases}.\label{ftilde}
	\end{equation}
	Then we have
	\begin{equation}
		\begin{split}
			\left| \mathbb{R}^{2}\setminus A_{\omega(t)}\right| &=\left| \mathbb{R}^{2}\setminus A_{\omega_{0}}\right| =\left| \left\lbrace x\in\mathbb{R}^{2} : \omega_{0}(x)> M+1\right\rbrace \right| =\int_{\left\lbrace \omega_{0}-M> 1\right\rbrace}1 dx\\
			&\leq\int_{\left\lbrace \omega_{0}-M> 1\right\rbrace}\left| \omega_{0}-M\right| dx\leq\int_{\left\lbrace \omega_{0}-M> 1\right\rbrace}\left| \omega_{0}-\zeta\right| dx\leq\left\| \omega_{0}-\zeta\right\| _{L^{1}},\quad t\geq0,\label{A^C}
		\end{split}		
	\end{equation}
	where the first equality holds
due to the conservation of level set measure \eqref{intro_level}.
Let's fix $ t\geq0 $ and for simplicity, we drop the parameter $ t $;
	\begin{equation}
		\nonumber \omega=\omega(t),\quad A_{\omega}=A_{\omega(t)},\quad \widetilde{\omega}=\widetilde{\omega(t)}.
	\end{equation}
	Then we have
	\begin{equation}
		\begin{split}
			\nonumber \left\| \omega-\zeta\right\| _{L^{1}}&=\int_{A_{\omega}}\left| \widetilde{\omega}-\zeta\right| dx+\int_{\mathbb{R}^{2}\setminus A_{\omega}}\left| \omega-\zeta\right| dx\\
			&\leq\left\| \widetilde{\omega}-\zeta\right\|_{L^{1}}+\int_{\mathbb{R}^{2}\setminus A_{\omega}}\omega dx+\int_{\mathbb{R}^{2}\setminus A_{\omega}}\zeta dx\\
			&\leq\underbrace{\left\| \widetilde{\omega}-(\widetilde{\omega})^{\ast}\right\|_{L^{1}} }_{(I)}+\underbrace{\left\| (\widetilde{\omega})^{\ast}-\zeta\right\|_{L^{1}} }_{(II)}+\underbrace{\int_{\mathbb{R}^{2}\setminus A_{\omega_{0}}}\omega_{0} dx}_{(III)}+\underbrace{\int_{\mathbb{R}^{2}\setminus A_{\omega}}\zeta dx}_{(IV)}.
		\end{split}		
	\end{equation}
	To estimate $ (III) $ and $ (IV) $, using the estimate \eqref{A^C} and the H\"{o}lder's inequality, we have
	\begin{align}
		\nonumber (III)&\leq\int_{\mathbb{R}^{2}\setminus A_{\omega_{0}}}\left| \omega_{0}-\zeta\right|  dx+\int_{\mathbb{R}^{2}\setminus A_{\omega_{0}}}\zeta dx\\
		\nonumber&\leq\left\| \omega_{0}-\zeta\right\| _{L^{1}}+\left\| \zeta\right\| _{L^{2}}\left| \mathbb{R}^{2}\setminus A_{\omega_{0}}\right| ^{\frac{1}{2}}\leq\left\| \omega_{0}-\zeta\right\| _{L^{1}}+\left\| \zeta\right\| _{L^{2}}\left\| \omega_{0}-\zeta\right\| _{L^{1}}^{\frac{1}{2}},		
	\end{align}
	\begin{align}
		\nonumber (IV)&\leq\left\| \zeta\right\| _{L^{2}}\left| \mathbb{R}^{2}\setminus A_{\omega}\right| ^{\frac{1}{2}}=\left\| \zeta\right\| _{L^{2}}\left| \mathbb{R}^{2}\setminus A_{\omega_{0}}\right| ^{\frac{1}{2}}\leq\left\| \zeta\right\| _{L^{2}}\left\| \omega_{0}-\zeta\right\| _{L^{1}}^{\frac{1}{2}}.
	\end{align}
	Thus, we get
	\begin{align}
		(III)+(IV)&\leq\left\| \omega_{0}-\zeta\right\| _{L^{1}}+2\left\| \zeta\right\| _{L^{2}}\left\| \omega_{0}-\zeta\right\| _{L^{1}}^{\frac{1}{2}}\leq\left\| \omega_{0}-\zeta\right\| _{L^{1}}+2\sqrt{M\alpha}\left\| \omega_{0}-\zeta\right\| _{L^{1}}^{\frac{1}{2}}.\label{(III)+(IV)}
	\end{align}
	To estimate $ (II) $, we observe the fact
	\begin{equation}
		\omega^{\ast}=(\omega_{0})^{\ast},\label{star}		
	\end{equation}
	because both $ \omega $ and $ \omega_{0} $ have the same measure for each level set. In other words, functions with same level set measure have the same rearrangement. Similarly, we have
	\begin{equation}
		(\widetilde{\omega})^{\ast}=(\widetilde{\omega_{0}})^{\ast},\label{tildestar}
	\end{equation}
	because the level set of both $ \widetilde{\omega} $ and $ \widetilde{\omega_{0}} $ is the empty set for $ a\geq M+1 $ while for the case $ a<M+1 $, we get
	\begin{equation}
		\begin{split}
			\nonumber\left| \left\lbrace x\in\mathbb{R}^{2} : \widetilde{\omega}(x)>a\right\rbrace \right|&=\left| \left\lbrace x\in\mathbb{R}^{2} : \omega(x)>a\right\rbrace \right| \\
			&=\left| \left\lbrace x\in\mathbb{R}^{2} : \omega_{0}(x)>a\right\rbrace \right| =\left| \left\lbrace x\in\mathbb{R}^{2} : \widetilde{\omega_{0}}(x)>a\right\rbrace \right| .
		\end{split}
	\end{equation}
	Then the nonexpansivity estimate \eqref{nonexp} in Lemma \ref{lemnonexp} and \eqref{tildestar} gives
	\begin{align}
		\begin{split}
			(II)&=\Big\| (\widetilde{\omega_{0}})^{\ast}-\underbrace{\zeta}_{=\zeta^{\ast}}\Big\|_{L^{1}}\leq\left\| \widetilde{\omega_{0}}-\zeta\right\|_{L^{1}}=\int_{A_{\omega_{0}}}\left| \omega_{0}-\zeta\right| dx+\int_{\mathbb{R}^{2}\setminus A_{\omega_{0}}}\left| (M+1)-\zeta\right| dx\\
			&\leq\int_{A_{\omega_{0}}}\left| \omega_{0}-\zeta\right| dx+\int_{\mathbb{R}^{2}\setminus A_{\omega_{0}}}\left| \omega_{0}-\zeta\right| dx=\left\| \omega_{0}-\zeta\right\|_{L^{1}}.\label{(II)}
		\end{split}		
	\end{align}	
	This is the first time we use the fact that $ \zeta $ is radially symmetric and non-increasing so that we have $ \zeta^{\ast}=\zeta $.
	
	To estimate $ (I) $, we recall $ \left\| \widetilde{\omega}\right\| _{L^{\infty}}\leq M+1 $. Thus, the rearrangement estimate \eqref{keyestimate} in Lemma \ref{lem2} says
	\begin{align}
		\nonumber (I)&\leq\sqrt{4\pi(M+1)}\bigg[J\left( \widetilde{\omega}\right) -J\left( (\widetilde{\omega})^{\ast}\right)\bigg]^{\frac{1}{2}}.
	\end{align}
	Then we split $ \Big(J\left( \widetilde{\omega}\right) -J\left( (\widetilde{\omega})^{\ast}\right)\Big) $ into 4 terms by adding and subtracting suitable terms:
	\begin{equation}
		\nonumber J\left( \widetilde{\omega}\right) -J\left( (\widetilde{\omega})^{\ast}\right)=\underbrace{\left[ J\left( \widetilde{\omega}\right) -J\left( \omega\right)\right]}_{(I_{a})}+\underbrace{\left[ J\left( \omega\right) -J\left( \zeta\right)\right] }_{(I_{b})}+\underbrace{\left[ J\left( \zeta\right) -J\left( \omega^{\ast}\right)\right] }_{(I_{c})} +\underbrace{\left[ J\left( \omega^{\ast}\right) -J\left( (\widetilde{\omega})^{\ast}\right)\right]}_{(I_{d})} .
	\end{equation}
	The term $ (I_{a}) $ is nonpositive, due to $ \widetilde{\omega}\leq \omega $, so it can be dropped. The estimate of $ (I_{b}) $ follows from the conservation of the angular impulse of $ \omega $:
	\begin{align}
		\nonumber (I_{b})&=J(\omega_{0})-J(\zeta)\leq J(|\omega_{0}-\zeta|).
	\end{align}
	The estimate of $ (I_{c}) $ becomes, using the equation \eqref{star},
	\begin{align}
		\nonumber (I_{c})&=J(\zeta)-J((\omega_{0})^{\ast})\leq\int_{B_{R}}|x|^{2}\left| \zeta-(\omega_{0})^{\ast}\right| dx+\int_{\mathbb{R}^{2}\setminus B_{R}}|x|^{2}\left| \zeta-(\omega_{0})^{\ast}\right| dx\\
		\nonumber&\leq R^{2}\int_{B_{R}}\left| \zeta-(\omega_{0})^{\ast}\right| dx+\int_{\mathbb{R}^{2}\setminus B_{R}} |x|^{2}\zeta dx+\int_{\mathbb{R}^{2}\setminus B_{R}}|x|^{2}(\omega_{0})^{\ast}dx.
	\end{align}
	For the last term, we shall use the following decomposition;
	\begin{equation*}
		\int_{\mathbb{R}^{2}\setminus B_{R}}|x|^{2}(\omega_{0})^{\ast}dx=\Big( J((\omega_{0})^{\ast})-J(\zeta) \Big)-\int_{B_{R}}|x|^{2} \Big((\omega_{0})^{\ast}-\zeta \Big)dx+\int_{\mathbb{R}^{2}\setminus B_{R}}|x|^{2}\zeta dx,\label{tail}
	\end{equation*}
	which gives us
	\begin{align}
		\nonumber (I_{c})&\leq R^{2}\int_{B_{R}}\left| \zeta-(\omega_{0})^{\ast}\right| dx+ \Big(J((\omega_{0})^{\ast})-J(\zeta)\Big)-\int_{B_{R}}|x|^{2}  \Big((\omega_{0})^{\ast}-\zeta \Big)dx+2\int_{\mathbb{R}^{2}\setminus B_{R}}|x|^{2}\zeta dx\\
		\nonumber&\leq R^{2}\int_{B_{R}}\left| \zeta-(\omega_{0})^{\ast}\right| dx+\Big( J(\omega_{0})-J(\zeta)\Big)+\int_{B_{R}}|x|^{2}  \left| (\omega_{0})^{\ast}-\zeta \right| dx+2\int_{\mathbb{R}^{2}\setminus B_{R}}|x|^{2}\zeta dx\\
		\nonumber&\leq 2R^{2}\int_{\mathbb{R}^{2}}\left| \zeta-(\omega_{0})^{\ast}\right| dx+J(|\omega_{0}-\zeta|)+2\int_{\mathbb{R}^{2}\setminus B_{R}}|x|^{2}\zeta dx.
	\end{align}
	Then by the nonexpansivity estimate \eqref{nonexp} in Lemma \ref{lemnonexp}, we have	
	\begin{align}
		\nonumber (I_{c})&\leq 2R^{2}\Big\| (\omega_{0})^{\ast}-\underbrace{\zeta}_{=\zeta^{\ast}}\Big\| _{L^{1}}+J(|\omega_{0}-\zeta|)+2\int_{\mathbb{R}^{2}\setminus B_{R}}|x|^{2}\zeta dx\\
		\nonumber&\leq 2R^{2}\left\| \omega_{0}-\zeta\right\| _{L^{1}}+J(|\omega_{0}-\zeta|)+2\int_{\mathbb{R}^{2}\setminus B_{R}}|x|^{2}\zeta dx.
	\end{align}	
	To estimate $ (I_{d}) $, we observe the fact
	\begin{equation}
		(\widetilde{\omega_{0}})^{\ast}=\widetilde{(\omega_{0})^{\ast}},\label{startilde}
	\end{equation}
	because both functions $ (\widetilde{\omega_{0}})^{\ast} $ and $ \widetilde{(\omega_{0})^{\ast}} $ are radially symmetric and non-increasing with the same measure for each level set. Indeed, if we have $ a\geq M+1 $, then each level set of both functions is the empty set and when $ a<M+1 $, we have, from the conservation of level set measure \eqref{levelsetmeas},
	\begin{equation}
		\begin{split}
			\nonumber \left| \left\lbrace x\in\mathbb{R}^{2} : (\widetilde{\omega_{0}})^{\ast}(x)>a\right\rbrace \right| &=\left| \left\lbrace x\in\mathbb{R}^{2} : \widetilde{\omega_{0}}(x)>a\right\rbrace \right| =\left| \left\lbrace x\in\mathbb{R}^{2} : \omega_{0}(x)>a\right\rbrace \right|\\
			&=\left| \left\lbrace x\in\mathbb{R}^{2} : (\omega_{0})^{\ast}(x)>a\right\rbrace \right| =\left| \left\lbrace x\in\mathbb{R}^{2} : \widetilde{(\omega_{0})^{\ast}}(x)>a\right\rbrace \right| .
		\end{split}
	\end{equation}
	Now using equations \eqref{star}, \eqref{tildestar}, and \eqref{startilde}, we have
	\begin{align}
		\nonumber (I_{d})&=J\left( (\omega_{0})^{\ast}\right) -J\left( (\widetilde{\omega_{0}})^{\ast}\right)=J\left( (\omega_{0})^{\ast}\right) -J\left( \widetilde{(\omega_{0})^{\ast}}\right).
	\end{align}
	Then splitting the integral range into $ A_{(\omega_{0})^{\ast}} $ and $ \mathbb{R}^{2}\setminus A_{(\omega_{0})^{\ast}} $ for each of the last terms above, the integral terms on $ A_{(\omega_{0})^{\ast}} $ get cancelled out;
	\begin{align}
		\begin{split}
			\nonumber (I_{d})=&\int_{A_{(\omega_{0})^{\ast}}}|x|^{2}\underbrace{(\omega_{0})^{\ast}}_{=\widetilde{(\omega_{0})^{\ast}}}dx+\int_{\mathbb{R}^{2}\setminus A_{(\omega_{0})^{\ast}}}|x|^{2}(\omega_{0})^{\ast}dx\\
			&-\left(  \int_{A_{(\omega_{0})^{\ast}}}|x|^{2}\widetilde{(\omega_{0})^{\ast}}dx+\int_{\mathbb{R}^{2}\setminus A_{(\omega_{0})^{\ast}}}|x|^{2}\widetilde{(\omega_{0})^{\ast}}dx\right) \leq\int_{\mathbb{R}^{2}\setminus A_{(\omega_{0})^{\ast}}}|x|^{2}(\omega_{0})^{\ast}dx.
		\end{split}
	\end{align}
	Then using $ \mathbb{R}^{2}\setminus A_{(\omega_{0})^{\ast}}=B_{\overline{r}} $ for some $ \overline{r}\geq0 $, the conservation of level set measure \eqref{levelsetmeas}, and the estimate \eqref{A^C}, we have
	\begin{equation}
		\begin{split}
			\nonumber \pi \overline{r}^{2}&=|\mathbb{R}^{2}\setminus A_{(\omega_{0})^{\ast}}|=\left| \left\lbrace y\in\mathbb{R}^{2} : (\omega_{0})^{\ast}>M+1\right\rbrace \right|=\left| \left\lbrace y\in\mathbb{R}^{2} : \omega_{0}>M+1\right\rbrace \right|\\
			&=|\mathbb{R}^{2}\setminus A_{\omega_{0}}|\leq\left\| \omega_{0}-\zeta\right\| _{L^{1}}.
		\end{split}
	\end{equation}
	Using this and the conservation of $ L^{1} $-norm \eqref{Lp_preserve}, we have
	\begin{equation}
		\begin{split}
			\nonumber (I_{d})&\leq \overline{r}^{2}\int_{B_{\overline{r}}}(\omega_{0})^{\ast}dx\leq \frac{1}{\pi}\left\| \omega_{0}-\zeta\right\| _{L^{1}}\left\| (\omega_{0})^{\ast}\right\| _{L^{1}}=\frac{1}{\pi}\left\| \omega_{0}-\zeta\right\| _{L^{1}}\left\| \omega_{0}\right\| _{L^{1}}\\
			&\leq \frac{1}{\pi} \left\| \omega_{0}-\zeta\right\| _{L^{1}}^{2}+\frac{1}{\pi}\left\| \zeta\right\| _{L^{1}}\left\| \omega_{0}-\zeta\right\| _{L^{1}}\leq \frac{1}{\pi} \left\| \omega_{0}-\zeta\right\| _{L^{1}}^{2}+\frac{\alpha}{\pi}\left\| \omega_{0}-\zeta\right\| _{L^{1}}.
		\end{split}		
	\end{equation}
	Now gathering all the estimates of $ (I_{a}) $, $ (I_{b}) $, $ (I_{c}) $, and $ (I_{d}) $, we have the following estimate for $ (I) $;
	\begin{equation}
		\begin{split}
			(I)&\leq\sqrt{4\pi(M+1)}\bigg[(I_{a})+(I_{b})+(I_{c})+(I_{d})\bigg]^{\frac{1}{2}}\\
			&\leq C_{R,M,\alpha}\Bigl[ \left\| \omega_{0}-\zeta\right\| _{L^{1}}^{\frac{1}{2}}+\left\| \omega_{0}-\zeta\right\| _{L^{1}}+J(|\omega_{0}-\zeta|)^{\frac{1}{2}}\Bigr]+C_{M}\left( \int_{\mathbb{R}^{2}\setminus B_{R}}|x|^{2}\zeta dx\right) ^{\frac{1}{2}} ,\label{(I)}
		\end{split}
	\end{equation}
	where $ C_{R,M,\alpha}>0 $ is a constant depending only on $ R,M,\alpha $ and $ C_{M}>0 $ is a constant depending only on $ M $.
	Finally, we have the estimate \eqref{L1} from estimates \eqref{(III)+(IV)} of $ (III)+(IV) $, \eqref{(II)} of $ (II) $, and \eqref{(I)} of $ (I) $.
\end{proof}

\subsection{Stability in $J_p$}

Before proving our main theorems, we first obtain $ {J}_{p} $-stability.

\begin{lem}\label{lemmain}
	Let $ R,M,\alpha\in(0,\infty) $. Then there exist constants $ C_{3}=C_{3}(R,M,\alpha)>0 $ and $ C_{4}=C_{4}(M)>0 $ such that if a function $ \zeta\in L^{\infty}(\mathbb{R}^{2}) $ with $ J(\zeta)<\infty $ is nonnegative, radially symmetric, non-increasing, and satisfies
	\begin{equation}
		\nonumber \left\| \zeta\right\| _{L^{\infty}}\leq M,\quad \left\| \zeta\right\| _{L^{1}}\leq\alpha,
	\end{equation}
	then for any nonnegative $ \omega_{0}\in L^{\infty}(\mathbb{R}^{2}) $ with $ J(\omega_{0})<\infty $, the solution $ \omega(t) $ of \eqref{eq1} satisfies
	\begin{equation}
		\begin{split}
			\sup_{t\geq0}\left\| \omega(t)-\zeta\right\| _{{J}_{p}}\leq &\;C_{3}\Bigr[ \left\| \omega_{0}-\zeta\right\| _{{J}_{p}}^{\frac{1}{2p}}+\left\| \omega_{0}-\zeta\right\| _{{J}_{p}}\Bigr] +R^{2}\cdot C_{4}\left( \int_{\mathbb{R}^{2}\setminus B_{R}}|x|^{2}\zeta dx\right) ^{\frac{1}{2}}\\
			&+C_{4}\Biggl[\left( \int_{\mathbb{R}^{2}\setminus B_{R}}|x|^{2}\zeta dx\right)^{\frac{1}{2p}}+ \int_{\mathbb{R}^{2}\setminus B_{R}}|x|^{2}\zeta dx\Biggr]\quad\text{ for any }\quad p\in[1,\infty)\label{mainestimate}
		\end{split}
	\end{equation}
\end{lem}

\begin{proof}
	As in the proof of the previous lemma, we fix $ t\geq0 $ and drop the parameter $ t $ for simplicity;
	\begin{equation}
		\nonumber \omega=\omega(t).
	\end{equation}
	To begin with, we use the decomposition
	\begin{equation}
		\nonumber\int_{\mathbb{R}^{2}\setminus B_{R}}|x|^{2}\omega dx=\Bigl( J(\omega)-J(\zeta)\Bigr) -\int_{B_{R}}|x|^{2}\left( \omega-\zeta\right) dx+\int_{\mathbb{R}^{2}\setminus B_{R}}|x|^{2}\zeta dx,
	\end{equation}	
	to get	
	\begin{align}
		\nonumber J(|\omega-\zeta|)&=\int_{B_{R}}|x|^{2}|\omega-\zeta|dx+\int_{\mathbb{R}^{2}\setminus B_{R}}|x|^{2}|\omega-\zeta|dx\\
		\nonumber&\leq R^{2}\int_{B_{R}}|\omega-\zeta|dx+\int_{\mathbb{R}^{2}\setminus B_{R}}|x|^{2}\omega dx+\int_{\mathbb{R}^{2}\setminus B_{R}}|x|^{2}\zeta dx\\
		\nonumber&\leq R^{2}\int_{B_{R}}|\omega-\zeta|dx+ \Big(\underbrace{J(\omega)}_{=J(\omega_{0})}-J(\zeta)\Big)+\int_{B_{R}}|x|^{2}| \omega-\zeta| dx+2\int_{\mathbb{R}^{2}\setminus B_{R}}|x|^{2}\zeta dx,
	\end{align}		
	which gives us
	\begin{equation}
		\begin{split}
			J(|\omega-\zeta|)&\leq 2R^{2}\int_{B_{R}}|\omega-\zeta|dx+J(|\omega_{0}-\zeta|)+2\int_{\mathbb{R}^{2}\setminus B_{R}}|x|^{2}\zeta dx\\
			&\leq 2R^{2}\left\| \omega-\zeta\right\| _{L^{1}}+J(|\omega_{0}-\zeta|)+2\int_{\mathbb{R}^{2}\setminus B_{R}}|x|^{2}\zeta dx.\label{JR^2}
		\end{split}		
	\end{equation}
	Then combined with the $ L^{1} $-estimate \eqref{L1} in Lemma \ref{lem3}, we have
	\begin{equation}
		\begin{split}
			J(|\omega-\zeta|)\leq &\;C'\Bigl[ \left\| \omega_{0}-\zeta\right\| _{L^{1}}^{\frac{1}{2}}+\left\| \omega_{0}-\zeta\right\| _{L^{1}}+J(|\omega_{0}-\zeta|)^{\frac{1}{2}}+J(|\omega_{0}-\zeta|)\Bigr]\\
			&+R^{2}C''\left( \int_{\mathbb{R}^{2}\setminus B_{R}}|x|^{2}\zeta dx\right)^{\frac{1}{2}}+C''\int_{\mathbb{R}^{2}\setminus B_{R}}|x|^{2}\zeta dx,\label{Jp=1}
		\end{split}
	\end{equation}
	where $ C'=C'(R,M,\alpha)>0 $ and $ C''=C''(M)>0 $ are constants. 
	Then adding \eqref{L1} from Lemma \ref{lem3} and \eqref{Jp=1} gives us \eqref{mainestimate} when $ p=1 $.
		
	Now we assume $ p\in(1,\infty) $. We may change constants
$ C'=C'(R,M,\alpha)>0 $ and $ C''=C''(M)>0 $   line by line, but they are remained independent of the choice of $p$.
	Then applying the estimate
	\begin{equation}
		\left\|g\right\| _{L^{1}}\leq\pi^{\frac{p-1}{p}}\left\|g\right\| _{L^{p}}+J(|g|)\leq\pi\left\|g\right\| _{L^{p}}+J(|g|),\quad g\in L^{p},\quad J(|g|)<\infty,\label{gLpJ}
	\end{equation}
	on \eqref{L1}, we get
	\begin{equation}
		\begin{split}
			\left\| \omega-\zeta\right\| _{L^{1}}\leq &\;C'\Bigl[ \left\| \omega_{0}-\zeta\right\| _{L^{p}}^{\frac{1}{2}}+\left\| \omega_{0}-\zeta\right\| _{L^{p}}+J(|\omega_{0}-\zeta|)^{\frac{1}{2}}+J(|\omega_{0}-\zeta|)\Bigr]+C''\left( \int_{\mathbb{R}^{2}\setminus B_{R}}|x|^{2}\zeta dx\right) ^{\frac{1}{2}}.\label{L1toLpJ}
		\end{split}
	\end{equation}
	Then using \eqref{L1toLpJ} on \eqref{JR^2}, we have
	\begin{equation}
		\begin{split}
			J(|\omega-\zeta|)\leq &\;C'\Bigl[ \left\| \omega_{0}-\zeta\right\| _{L^{p}}^{\frac{1}{2}}+\left\| \omega_{0}-\zeta\right\| _{L^{p}}+J(|\omega_{0}-\zeta|)^{\frac{1}{2}}+J(|\omega_{0}-\zeta|)\Bigr]\\
			&+R^{2}C''\left( \int_{\mathbb{R}^{2}\setminus B_{R}}|x|^{2}\zeta dx\right)^{\frac{1}{2}}+C''\int_{\mathbb{R}^{2}\setminus B_{R}}|x|^{2}\zeta dx.\label{L1w}
		\end{split}
	\end{equation}
	To get the $ L^{p} $ estimate, we use the set $ A_{\omega} $ and the function $ \widetilde{\omega} $ defined by \eqref{Af} and \eqref{ftilde} with $ f=\omega $ from the proof of Lemma \ref{lem3} once more. First, we have
	\begin{equation}
		\begin{split}
			\nonumber \left\| \omega-\zeta\right\| _{L^{p}}^{p}&=\int_{A_{\omega}}|\underbrace{\omega}_{=\widetilde{\omega}}-\zeta|^{p}dx+\int_{\mathbb{R}^{2}\setminus A_{\omega}}|\omega-\zeta|^{p}dx\\
			&\leq\left\| \widetilde{\omega}-\zeta\right\| _{L^{\infty}}^{p-1}\int_{A_{\omega}}|\omega-\zeta|dx+2^{p}\left( \int_{\mathbb{R}^{2}\setminus A_{\omega}}|\omega|^{p}dx+\int_{\mathbb{R}^{2}\setminus A_{\omega}}|\zeta|^{p}dx\right) \\
			&\leq(M+1)^{p-1}\left\| \omega-\zeta\right\| _{L^{1}}+2^{p}\Biggl( \underbrace{\int_{\mathbb{R}^{2}\setminus A_{\omega_{0}}} |\omega_{0}|^{p}dx}_{(I)}+\underbrace{\int_{\mathbb{R}^{2}\setminus A_{\omega}} |\zeta|^{p} dx}_{(II)}\Biggr) .
		\end{split}
	\end{equation}
	We also borrow the estimate \eqref{A^C};
	\begin{equation}
		\nonumber\left| \mathbb{R}^{2}\setminus A_{\omega_{0}}\right| \leq\left\| \omega_{0}-\zeta\right\| _{L^{1}}.
	\end{equation}
	Then we get
	\begin{equation}
		\begin{split}
			\nonumber (I)&\leq2^{p}\left( \int_{\mathbb{R}^{2}\setminus A_{\omega_{0}}} |\omega_{0}-\zeta|^{p}dx+\int_{\mathbb{R}^{2}\setminus A_{\omega_{0}}} |\zeta|^{p}dx\right) \leq2^{p}\Bigl( \left\| \omega_{0}-\zeta\right\|_{L^{p}}^{p} +M^{p}\left| \mathbb{R}^{2}\setminus A_{\omega_{0}}\right| \Bigr) \\
			&\leq2^{p}\Bigl( \left\| \omega_{0}-\zeta\right\|_{L^{p}}^{p} +M^{p}\left\| \omega_{0}-\zeta\right\|_{L^{1}}\Bigr).
		\end{split}
	\end{equation}
	The estimate of $ (II) $ follows in the same way;
	\begin{equation}
		\nonumber (II)\leq M^{p}\left| \mathbb{R}^{2}\setminus A_{\omega}\right|=M^{p}\left| \mathbb{R}^{2}\setminus A_{\omega_{0}}\right| \leq M^{p}\left\| \omega_{0}-\zeta\right\|_{L^{1}}.
	\end{equation}
	Then we have
	\begin{equation}
		\begin{split}
			\nonumber\left\| \omega-\zeta\right\| _{L^{p}}^{p}&\leq  (M+1)^{p-1} \left\| \omega-\zeta\right\| _{L^{1}}+2^{p}(2^{p}+1)M^{p}\left\| \omega_{0}-\zeta\right\| _{L^{1}}+ 2^{2p}\left\| \omega_{0}-\zeta\right\|_{L^{p}}^{p}\\
			&\leq  (M+1)^{p} \left\| \omega-\zeta\right\| _{L^{1}}+2^{3p}M^{p}\left\| \omega_{0}-\zeta\right\| _{L^{1}}+ 2^{2p}\left\| \omega_{0}-\zeta\right\|_{L^{p}}^{p}
		\end{split}
	\end{equation}
	Thus, by estimates \eqref{gLpJ} and \eqref{L1toLpJ}, we have
	\begin{equation}
		\begin{split}
			\left\| \omega-\zeta\right\| _{L^{p}}\leq&\; C'\Big[\left\| \omega_{0}-\zeta\right\| _{L^{p}}^{\frac{1}{2p}}+\left\| \omega_{0}-\zeta\right\| _{L^{p}}+J(|\omega_{0}-\zeta|)^{\frac{1}{2p}}+J(|\omega_{0}-\zeta|)^{\frac{1}{p}}\Big]\\
			&+C''\left( \int_{\mathbb{R}^{2}\setminus B_{R}}|x|^{2}\zeta dx\right) ^{\frac{1}{2p}}.
		\end{split}\label{Lp}
	\end{equation}
	Finally, adding \eqref{L1w} and \eqref{Lp}, we obtain \eqref{mainestimate} for $ p\in(1,\infty) $.
\end{proof}

Now we are ready to prove our main theorems.

\begin{proof}[Proof of Theorem \ref{thm2}]
	We let $ R,M\in(0,\infty) $, $ p\in[1,\infty) $, and set $ \alpha:=\pi R^{2}M $. Also, we let $ \zeta\in L^{\infty}(\mathbb{R}^{2}) $ be nonnegative, radially symmetric, non-increasing, and compactly supported with \eqref{RM}. Additionally, we let $ \omega_{0}\in L^{\infty}(\mathbb{R}^{2}) $ be nonnegative with $ J(\omega_{0})<\infty $ and $ \omega(t) $ be the solution of \eqref{eq1} for $ \omega_{0} $. Then we have $ \left\| \zeta\right\| _{L^{1}}=\int_{B_{R}}\zeta dx\leq \pi R^{2}M\leq\alpha $, so we can use the conclusion \eqref{mainestimate} in Lemma \ref{lemmain}. Furthermore, we have
	\begin{equation}
		\nonumber\int_{\mathbb{R}^{2}\setminus B_{R}}|x|^{2}\zeta dx=0,
	\end{equation}
	which makes all the terms containing $ \int_{\mathbb{R}^{2}\setminus B_{R}}|x|^{2}\zeta dx $ on the right-hand side of \eqref{mainestimate} zero. In sum, we get
	\begin{equation}
		\nonumber\sup_{t\geq0}\left\| \omega(t)-\zeta\right\| _{{J}_{p}}\leq C_{3}\Bigr[ \left\| \omega_{0}-\zeta\right\| _{{J}_{p}}^{\frac{1}{2p}}+\left\| \omega_{0}-\zeta\right\| _{{J}_{p}}\Bigr] ,
	\end{equation}
	with $ C_{3}=C_{3}(R,M,\alpha)=C_{3}(R,M,\pi R^{2}M) $.
\end{proof}

\begin{proof}[Proof of Theorem \ref{thm3}]
	We let $ \varepsilon>0 $, $ p\in[1,\infty) $, and $ \zeta\in L^{\infty}(\mathbb{R}^{2}) $ be nonnegative, radially symmetric, non-increasing with
	\begin{equation}
		\int_{\mathbb{R}^{2}}|x|^{6}\zeta dx<\infty.\label{x^6}
	\end{equation}
	We set $ M,\alpha\in(0,\infty) $ by $ \left\| \zeta\right\| _{L^{\infty}}=M $, $ \left\| \zeta\right\| _{L^{1}}=\alpha $. (If we have $ \zeta\equiv0 $, then the conclusion is trivial.) Then we borrow the constant $ C_{4}=C_{4}(M) $ from Lemma \ref{lemmain}. We also let $ \omega_{0}\in L^{\infty}(\mathbb{R}^{2}) $ be nonnegative with $ J(\omega_{0})<\infty $ and $ \omega(t) $ be the solution of \eqref{eq1} for $ \omega_{0} $. Then due to $ J(\zeta)<\infty $ by \eqref{x^6}, there exists (large) $ R_{1}>0 $ such that
	\begin{equation}
		\nonumber C_{4}\Biggl[\left( \int_{\mathbb{R}^{2}\setminus B_{R}}|x|^{2}\zeta dx\right)^{\frac{1}{2p}}+\int_{\mathbb{R}^{2}\setminus B_{R}}|x|^{2}\zeta dx\Biggr]\leq\varepsilon\quad\text{ for all }\quad R\geq R_{1}.
	\end{equation}
	In addition, there exists (large) $ R_{2}>0 $ such that
	\begin{equation}
		\nonumber R^{2}\cdot C_{4}\left( \int_{\mathbb{R}^{2}\setminus B_{R}}|x|^{2}\zeta dx\right) ^{\frac{1}{2}}\leq\varepsilon\quad\text{ for all }\quad R\geq R_{2}.
	\end{equation}
	Indeed, this holds due to \eqref{x^6};
	\begin{equation}
		\nonumber R^{2}\cdot\left( \int_{\mathbb{R}^{2}\setminus B_{R}}|x|^{2}\zeta dx\right) ^{\frac{1}{2}}\leq \left( \int_{\mathbb{R}^{2}\setminus B_{R}}|x|^{6}\zeta dx\right) ^{\frac{1}{2}}\longrightarrow0\quad\text{ as }\quad R\longrightarrow\infty.
	\end{equation}
	We take $ R=\max\left\lbrace R_{1}, R_{2}\right\rbrace  $. Now we take the constant $ C_{3}=C_{3}(R,M,\alpha) $ from Lemma \ref{lemmain}. Finally, by taking (small) $ \delta>0 $ that satisfies $ C_{3}(\delta^{\frac{1}{2p}}+\delta)\leq\varepsilon $, if we have
	\begin{equation}
		\nonumber \left\| \omega_{0}-\zeta\right\| _{{J}_{p}}\leq\delta,
	\end{equation}
	then the estimate \eqref{mainestimate} in Lemma \ref{lemmain} says
	\begin{equation}
		\begin{split}
			\nonumber \sup_{t\geq0}\left\| \omega(t)-\zeta\right\| _{{J}_{p}}\leq &\;C_{3}\Bigr[ \left\| \omega_{0}-\zeta\right\| _{{J}_{p}}^{\frac{1}{2p}}+\left\| \omega_{0}-\zeta\right\| _{{J}_{p}}\Bigr] +R^{2}\cdot C_{4}\left( \int_{\mathbb{R}^{2}\setminus B_{R}}|x|^{2}\zeta dx\right) ^{\frac{1}{2}}\\
			&+C_{4}\Biggl[\left( \int_{\mathbb{R}^{2}\setminus B_{R}}|x|^{2}\zeta dx\right)^{\frac{1}{2p}}+\int_{\mathbb{R}^{2}\setminus B_{R}}|x|^{2}\zeta dx\Biggr]\\
			\leq &\;\varepsilon+\varepsilon+\varepsilon=3\varepsilon.
		\end{split}		
	\end{equation}
	
\end{proof}
 
{
 \section*{Acknowledgement}}

\noindent KC has been supported by   the National Research Foundation of Korea (NRF-2018R1D1A1B07043065).
KC thanks In-Jee Jeong for helpful discussions.

\ \\ 

\bibliography{Choi_Lim_20210322_references}

\end{document}